\documentclass[pdftex,a4paper,11pt,reqno]{amsart}

\usepackage{amsmath, amssymb, amsthm, mathtools}
\usepackage{mathrsfs} 
\usepackage{enumerate}

\usepackage{color}

\usepackage{comment}

\usepackage{verbatim}

\usepackage[pdftex]{graphicx}
\usepackage{float}

\usepackage{url}

\newcommand{\A}{\mathcal{A}}
\newcommand{\R}{\mathbb{R}}

\newcommand{\N}{\mathbb{N}}
\newcommand{\Z}{\mathbb{Z}}
\newcommand{\T}{\mathbb{T}}

\DeclareMathOperator{\BUC}{BUC}
\DeclareMathOperator{\USC}{USC}
\DeclareMathOperator{\LSC}{LSC}
\DeclareMathOperator{\Lip}{Lip}
\DeclareMathOperator{\AC}{AC}

\theoremstyle{plain}
\newtheorem{theorem}{Theorem}[section]
\newtheorem*{theorem*}{Theorem}
\newtheorem{lemma}[theorem]{Lemma}
\newtheorem*{lemma*}{Lemma}

\newtheorem*{corollary*}{Corollary}
\newtheorem{proposition}[theorem]{Proposition}
\newtheorem*{proposition*}{Proposition}

\newtheorem*{innercustomtheorem}{\customtheoremname}
\newcommand{\customtheoremname}{}

\theoremstyle{definition}
\newtheorem{definition}[theorem]{Definition}
\newtheorem*{definition*}{Definition}
\newtheorem{example}[theorem]{Example}
\newtheorem*{example*}{Example}
\newtheorem{remark}[theorem]{Remark}
\newtheorem*{remark*}{Remark}

\newtheorem*{notation*}{Notation}
\newtheorem*{innercustomdefinition}{\customdefinitionname}
\newcommand{\customdefinitionname}{}
\newenvironment{customdefinition}[1]
  {\renewcommand{\customdefinitionname}{#1}%
   \begin{innercustomdefinition}}
  {\end{innercustomdefinition}}

\newenvironment{subproof}[1][Proof]
  {%
   \begin{proof}[#1]}
  {\end{proof}}

\usepackage[
  pdftex,
  bookmarks=true,
  bookmarksnumbered=true,
  bookmarkstype=toc,
  hidelinks]{hyperref}

\usepackage{cleveref}
\crefformat{equation}{(#2#1#3)}

\crefname{theorem}{theorem}{theorems}
\Crefname{theorem}{Theorem}{Theorems}
\crefname{proposition}{proposition}{propositions}
\Crefname{proposition}{Proposition}{Propositions}
\crefname{lemma}{lemma}{lemmas}
\Crefname{lemma}{Lemma}{Lemmas}
\crefname{corollary}{corollary}{corollaries}
\Crefname{corollary}{Corollary}{Corollaries}
\crefname{definition}{definition}{definitions}
\Crefname{definition}{Definition}{Definitions}
\crefname{remark}{remark}{remarks}
\Crefname{remark}{Remark}{Remarks}
\crefname{example}{example}{examples}
\Crefname{example}{Example}{Examples}

\numberwithin{equation}{section}

\usepackage{pgfplots}
\pgfplotsset{compat=1.18}

\title[Quantitative homogenization for obstacle problems]
{Quantitative homogenization for obstacle problems of convex Hamilton--Jacobi equations}
\author[Takuya Sato]{Takuya Sato}
\author[Jingcheng Ye]{Jingcheng Ye}
\address[Takuya Sato]{Corresponding author. 
Graduate School of Mathematical Sciences, 
The University of Tokyo. 3-8-1 Komaba, 
Meguro-ku, Tokyo 153-8914, Japan.}
\email{satoh-t@g.ecc.u-tokyo.ac.jp}
\address[Jingcheng Ye]{KPMG Consulting Co., Ltd.
OTEMACHI FINANCIAL CITY South Tower,
1-9-7 Otemachi,
Chiyoda-ku, Tokyo 100-0004, Japan.}
\email{yejc.pde@gmail.com}
\date{\today}

\begin{document}

\begin{abstract}
We study the homogenization for obstacle problems of convex Hamilton--Jacobi equations in periodic media with obstacle functions depending on both macroscopic and microscopic spatial variables.
Based on the methods of analyzing variational representation formulas,
we establish the optimal convergence rate $O( \varepsilon )$ for the homogenization limit
under the assumption that the obstacle function is Lipschitz continuous.
Moreover, we show that this convergence rate may deteriorate due to the presence of the obstacle when the obstacle function is not Lipschitz continuous, by constructing an example.
\end{abstract}

\maketitle


\section{Introduction}
We study the homogenization problem of Hamilton--Jacobi equations with obstacles.
For each $\varepsilon >0$, 
let $u^\varepsilon$ be the unique viscosity solution to the obstacle problem of an evolutionary equation of the form:
\begin{align}\label{OP}
  \left\{
  \begin{aligned}
  &\max\left\{
  u^\varepsilon_t + H\left( \frac{x}{\varepsilon}, Du^\varepsilon \right),\ 
  u^\varepsilon - \psi \left(x, \frac{x}{\varepsilon} \right)
  \right\} = 0
  && \text{in } \R^n \times (0,\infty),\\
  &u^\varepsilon (x,0) = u_0(x)
  && \text{on } \R^n,
  \end{aligned}
  \right.\tag{$\mathrm{OP}_\varepsilon$}
\end{align}
where $H$ is a given Hamiltonian,
$\psi: \R^n \times \R^n \to \R$ is a given obstacle function
and $u_0$ is an initial condition.
Under appropriate assumptions for $H, \psi$ and $u_0$, 
we can expect that, as $\varepsilon \to +0$,
$u^\varepsilon$ converges locally uniformly 
to the function $u:\R^n\times[0,\infty)\to\R$ which is the unique viscosity solution 
to the obstacle problem:
\begin{align}\label{effOP}
    \left\{
    \begin{aligned}
        &\max\left\{
        u_t + \overline{H}(Du),\ 
        u-\overline{\psi}(x)
        \right\}=0
        &&\text{in }\R^n \times (0,\infty),\\
        &u(x,0)=u_0(x)
        &&\text{on } \R^n.
    \end{aligned}
    \right.\tag{$\overline{\mathrm{OP}}$}
\end{align}
Here, $\overline{H}\in C(\R^n)$ is the effective Hamiltonian
determined by the cell problem
\begin{align}\label{Cell}
    H(y,p+Dv)=\overline{H}(p) \quad \text{in } \T^n:= \R^n/\Z^n
    \tag{$\mathrm{Cell}_p$}
\end{align}
for each $p\in\R^n$
(more precisely, $\overline{H}(p)$ is the unique constant such that 
(\ref{Cell}) has a $\Z^n$-periodic viscosity solution $v$),
and the function $\overline{\psi}:\R^n\to\R$,
which we call \textit{the effective obstacle}, is defined by
\begin{align}
    \overline{\psi}(x) = \min_{y\in\R^n}\psi(x,y).
\end{align}
Throughout this paper, we assume, for (\ref{OP}), the Hamiltonian $H=H(y,p):\R^n\times \R^n \to \R$ satisfies:
\begin{enumerate}[$(\mathrm{A}1)$]
    \item $H\in C(\R^n\times\R^n)$ 
    and the map $y\mapsto H(y,p)$ is $\Z^n$-periodic for each $p\in\R^n$,
    \label{assump:H-peri}
    \item $H$ is coercive in $p$, that is, \label{assump:coer}
    \begin{align*}
        \lim_{|p|\to\infty} \inf_{y\in\R^n} H(y,p) = +\infty,
    \end{align*}
    \item the map $p\mapsto H(y,p)$ is convex for each $y\in\R^n$.\label{assump:convex}
\end{enumerate}
Additionaly, we assume that the obstacle function $\psi=\psi(x,y):\R^n\times\R^n\to\R$  and the initial datum $u_0:\R^n\to \R$ satisfy:
\begin{enumerate}[$(\mathrm{A}1)$]
\setcounter{enumi}{3}
    \item $\psi\in C(\R^n\times\R^n)$, $\psi$ is bounded below
    and the map $y\mapsto \psi(x,y)$ is $\Z^n$-periodic for each $x\in\R^n$,\label{assump:psi-peri}
    \item $u_0\in\BUC(\R^n)\cap\Lip(\R^n)$ and satisfies
    \begin{align*}
        u_0(x) \leq \min_{y\in\R^n} \psi(x,y) \ (\ =\overline{\psi}(x)\ ) \quad \text{for all } x\in\R^n.
    \end{align*}
    \label{assump:initial}
\end{enumerate}

\subsection*{Obstacle problems for Hamilton--Jacobi equations}
Obstacle problems for Hamilton--Jacobi equations arise as PDEs satisfied by the value functions of optimal stopping problems in control theory. 
Let $L=L(y,q)$ be the Lagrangian determined by the following Legendre transform of $H$:
\begin{align*}
    L(y,q)
    =\sup_{p\in \R^n}
    \left\{
    p\cdot q - H(y,p)
    \right\},
    \quad (y,q)\in \R^n\times\R^n.
\end{align*}
Since $H(y,p)$ is $\Z^n$-periodic in $y$,
and convex and coercive in $p$,
it follows that $L(y,q)$ is also $\Z^n$-periodic in $y$,
and convex and coercive in $q$.
For each $\varepsilon>0$ and $(x,t)\in\R^n\times[0,\infty)$,
we consider the minimizing problem of the cost functional
\begin{align*}
    J^\varepsilon(x,t,\gamma,\theta)
    :=\int_\theta^t L\left(
    \frac{\gamma(s)}{\varepsilon},\dot{\gamma}(s)\right)ds
    +\mathbf{1}_{\{\theta=0\}}u_0(\gamma(0))
    +\mathbf{1}_{\{\theta>0\}}\psi\left(\gamma(\theta),\frac{\gamma(\theta)}{\varepsilon}\right),
\end{align*}
where the controller can choose any pair $(\gamma,\theta)$
of an admissible trajectry and a stopping time with
\begin{align*}
    \gamma \in \A(x,t):=\{\gamma\in \AC([0,t];\R^n)\; :\; \gamma(t)=x\}
    \quad\text{and}
    \quad 
    \theta \in [0,t],
\end{align*}
respectively.
Here, $\mathbf{1}_E:[0,t]\to\{0,1\}$ denotes the indicator function of the set $E\subset[0,t]$.
A distinctive feature of this optimal control problem is
that the controller may either continue to optimize the trajectories satisfying the terminal condition or stop the process at any time by paying the stopping cost $\psi$.
Then, we consider the value function $u^\varepsilon$ for this control problem determined by
\begin{align}
    u^\varepsilon(x,t) = \inf_{\substack{\gamma\in\mathcal{A}(x,t)\\\theta\in[0,t]}}
    J^\varepsilon(x,t,\gamma,\theta).
    \label{eq:value-fcn-intro}
\end{align}
One can show that this value function $u^\varepsilon$ satisfies the dynamic programming principle, that is, 
$u^\varepsilon$ satisfies
\begin{align*}
    u^\varepsilon(x,t)
    =\inf_{\substack{\gamma\in\mathcal{A}(x,t)\\\theta\in[0,t]}}\left\{
    \int_{\theta\vee s}^t L\left(\frac{\gamma(r)}{\varepsilon},\dot{\gamma}(r)\right)dr
    +\mathbf{1}_{\{\theta\leq s\}}u^\varepsilon(\gamma(s),s)
    +\mathbf{1}_{\{\theta> s\}}\psi \left(\gamma(\theta),\frac{\gamma(\theta)}{\varepsilon}\right)
    \right\}
\end{align*}
for all $(x,t)\in\R^n\times[0,\infty)$ and $s\in[0,t]$.
Here, 
we use the notation 
$a\vee b = \max\{ a, b\}$
and $a\wedge b =\min\{a, b\}$
for $a,b\in\R$.
Moreover,
this dynamic programming principle equality
implies that $u^\varepsilon$ is a viscosty solution to (\ref{OP}).

From the viewpoint of control, 
the obstacle represents the cost that the controller needs to pay for stopping the process, 
and the first equation in (\ref{OP}), which is called \textit{the Hamilton--Jacobi variational inequality of obstacle type}, describes the competition between continuation and stopping.
Optimal stopping problems form a fundamental class of optimal control problems 
whose value functions satisfy variational inequalities.
In particular, 
in the absence of control acting on the dynamics,
such problems were studied via the dynamic programming approach 
before the development of viscosity solution theory, 
and the analysis relied on the free boundary separating the continuation and stopping regions. 
We refer to \cite{MR592928, MR636737, MR651053} for classical results and further developments.
For the viscosity solution theory of Hamilton--Jacobi equations arising from optimal stopping problems, we refer to \cite{MR921827,MR1484411}.
In Section 2, 
we summarize some fundamental properties of viscosity solutions to evolutionary Hamilton--Jacobi equations with obstacles 
and include a complete proof of representation formula (\ref{eq:value-fcn-intro}).

\subsection*{Previous literatures for quantitative homogenization}
Homogenization problems of Hamilton--Jacobi equations in periodic media was initiated in the seminal work of Lions, 
Papanicolaou and Varadhan \cite{LPV87}, 
and has become an important topic in the theory of viscosity solutions with connections to optimal control, 
calculus of variations and dynamical systems. 
Consider the following Cauchy problem for a Hamilton--Jacobi equation: 
\begin{align}
    \left\{
  \begin{aligned}
  &
  u^\varepsilon_t + H\left( \frac{x}{\varepsilon}, Du^\varepsilon \right)= 0
  && \text{in } \R^n \times (0,\infty),\\
  &u^\varepsilon (x,0) = u_0(x)
  && \text{on } \R^n.
  \end{aligned}
  \right.\label{CP}
\end{align}
Under suitable assumptions on $H$ such as (A\ref{assump:H-peri}) and (A\ref{assump:coer}), and on $u_0$, it is known that the unique viscosity solution $u^\varepsilon$ to (\ref{CP}) converges, as $\varepsilon \to +0$, locally uniformly to a function $u$ which solves an effective Hamilton--Jacobi equation:
\begin{align}
    \left\{
  \begin{aligned}
  &
  u_t + \overline{H}\left(Du \right)= 0
  && \text{in } \R^n \times (0,\infty),\\
  &u(x,0) = u_0(x)
  && \text{on } \R^n.
  \end{aligned}
  \right.\label{effCP}
\end{align}
Here, 
$\overline{H}(p)$ is determined by cell problem (\ref{Cell}) and it is known that, 
the function $\overline{H}:\R^n\to\R$ is continuous and coercive, and if $H$ is convex, 
then $\overline{H}$ is also convex.
It was proved in \cite{LPV87} and \cite{MR1159184} 
that the unique viscosity solution $u^\varepsilon$ to (\ref{CP}) converges locally uniformly 
to the unique viscosity solution $u$ to (\ref{effCP}).
The convergence can be established by a PDE approach based on Evans’ perturbed test function method,
a viscosity solution technique that combines the doubling-of-variables argument with correctors $v$ obtained from the cell problem.

Next, we review some representative results on quantitative homogenization theory for Hamilton--Jacobi equations.
Quantitative homogenization aims to estimate the rate of convergence of $u^\varepsilon$ to a homogenized solution $u$.
For general coercive Hamiltonians, 
Capuzzo\nobreakdash-Dolcetta and Ishii \cite{MR1871349} obtained the convergence rate $O(\varepsilon^{1/3})$ for the homogenization of (\ref{CP}).
In the convex setting,
Mitake, Tran, and Yu \cite{MR3951696} developed a quantitative approach based on representation formulas 
from optimal control theory and fine properties of minimizing curves,
making essential use of backward characteristics.
From the same viewpoint,
Tran and Yu \cite{MR4946874} further refined the analysis and established the optimal convergence rate $O(\varepsilon)$ 
for periodic homogenization of convex Hamilton--Jacobi equations.
A key ingredient in \cite{MR4946874} is the use of a metric-type quantity $m(t, x, y)$, defined as
\begin{align*}
    m(t,x,y)=\inf\left\{\int_0^t L(\eta(s),\dot{\eta}(s))ds\; :\; \eta\in \AC ([0,t];\R^n),\;
    \eta(0)=x,\;\eta(t)=y\right\}.
\end{align*}
From the viewpoint of optimal control, 
$m(t,x,y)$ represents the minimum cost 
traveling from $x\in\R^n$ to $y\in\R^n$ in a fixed time $t>0$. 
They established an approximate sub and superadditivity property for $m$. 
More precisely, 
they proved that for each $M>0$,
there exists $C>0$ depending only on $n$, $L$ and $M$ such that 
\begin{align*}
    2m(t,x,y)-C \leq m(2t,2x,2y) \leq 2m(t,x,y) + C
\end{align*}
holds for every $x,y$ and $t$ with $|x-y|\leq Mt$. 
Subadditivity of $m$ implies, 
via Fekete’s lemma, 
that the homogenized limit 
\begin{align*}
    \overline{m}(t,x,y):=
    \lim_{k\to\infty} \frac{1}{k}m(kt,kx,ky)
\end{align*}
exists for each $x,y$ and $t$. 
Moreover, combining this with superadditivity, it follows that: for every $M>0$,
there exists a constant $C>0$ depending only on $n$,
$L$ and $M$ such that 
\begin{align}
    \left|
        \varepsilon m\left(
            \frac{t}{\varepsilon},
            \frac{x}{\varepsilon},
            \frac{y}{\varepsilon}
        \right)
        -\overline{m}(t,x,y)
    \right|
    \leq C\varepsilon
    \label{eq:conv-rate-metric}
\end{align}
holds for every $x,y\in\R^n$ and $t>0$ with $|x-y|\leq Mt$.

The proof of the superadditivity property relies on a geometric construction, 
often referred to as a curve surgery argument:
roughly speaking, 
the idea is to cut an optimal trajectory via a curve cutting lemma due to topological argument of Burago \cite{MR1279391}, 
shift the resulting pieces using periodicity,
and patch them back together to build admissible competitors.

Using the optimal control representation for Cauchy problems together with a change of variables,
the unique viscosity solution $u^\varepsilon$ to (\ref{CP}) can be rewritten in the form
\begin{align}
    u^\varepsilon (x,t)
    =\inf_{z\in\R^n}\left\{
    \varepsilon m\left(\frac{t}{\varepsilon},\frac{z}{\varepsilon},\frac{x}{\varepsilon}\right)+u_0(z)
    \right\}.\label{eq:metric-rep}
\end{align}
Combining the homogenization result for (\ref{CP}) with the Hopf--Lax formula for (\ref{effCP}), 
one obtains the following representation formula for the homogenized solution $u$ of (\ref{effCP}):
\begin{align}
    u(x,t)
    =\inf_{z\in\R^n}\left\{ t\overline{L}\left(
    \frac{x-z}{t}
    \right)+u_0(z) \right\}
    =\inf_{z\in\R^n}\left\{
    \overline{m}(t,z,x) + u_0(z)
    \right\}.\label{eq:homo-metric-rep}
\end{align}
The estimate $\lVert u^\varepsilon -u\rVert_{L^\infty (\R^n\times[0,\infty))}\leq C\varepsilon$ then follows by applying (\ref{eq:conv-rate-metric}) to (\ref{eq:metric-rep}) and (\ref{eq:homo-metric-rep}). 
This observation forms the basis of the proof of the optimal convergence rate $O(\varepsilon)$. 

Building on the optimal control approach,
sharp convergence rates have recently been established in a variety of homogenization settings for convex Hamilton--Jacobi equations.
The quantitative theory has subsequently been extended to Cauchy problems with multiscale Hamiltonians $H(x,x/\varepsilon,p)$ \cite{MR4198478, MR4646025}, 
$t/\varepsilon$-variable dependent Hamiltonians \cite{MR4765431}, 
$u/\varepsilon$-variable dependent Hamiltonians \cite{MNT25},
weakly coupled Hamilton--Jacobi systems \cite{MR4911666}, 
viscous Hamilton--Jacobi equations \cite{MR4828487}, among others.
Quantitative homogenization for initial--boundary value problems on perforated domains has also been investigated. 
The metric approach of \cite{MR4946874} 
was extended to state-constraint, Neumann, and Dirichlet boundary conditions on perforated domains 
in \cite{MR4870323}, \cite{MR5053934}, and \cite{HT25}, respectively. 
On the other hand, 
quantitative homogenization and sharp error estimates 
for obstacle problems are still scarce,
although variational inequalities of obstacle type naturally arise in optimal control theory. 
Our present work is motivated by this gap in the literature,
as well as by the recent quantitative studies on perforated domains \cite{MR4870323, MR5053934, HT25}.

\subsection*{Main results}

The main theorem in this paper is the following quantitative homogenization result for (\ref{OP}).

\begin{theorem}\label{thm:main-thm}
    Assume $(\mathrm{A}\ref{assump:H-peri})$--$(\mathrm{A}\ref{assump:initial})$.
    Let $u^\varepsilon$ and $u$ be the unique viscosity solutions 
    to $($\ref{OP}$)$ for each $\varepsilon>0$ 
    and $($\ref{effOP}$)$, respectively.
    Then, as $\varepsilon \to +0$,
    $u^\varepsilon$ converges locally uniformly on $\R^n\times [0,\infty)$ to $u$.
    Moreover, if $\psi\in\Lip(\R^n\times\R^n)$,
    then there exists $C>0$ depending only on $H$, $\psi$ and $u_0$ such that
    \begin{align*}
        \lVert u^\varepsilon - u \rVert_{L^\infty(\R^n\times (0,\infty))}
        \leq C\varepsilon
        \quad \text{for all } \varepsilon >0.
    \end{align*}
\end{theorem}

Furthermore, in Section 3.2,
we show that the Lipschitz continuity assumption on $\psi$ in Theorem \ref{thm:main-thm} is essential for obtaining the convergence rate $O(\varepsilon)$.
More precisely, 
we construct an example of (\ref{OP}) with a non-Lipschitz obstacle function $\psi$, 
for which the convergence rate to the homogenized limit is strictly slower than $O(\varepsilon)$.

In addition, 
in Section 2.1, 
we provide a proof of the unique solvability 
and an a priori Lipschitz estimate 
for viscosity solutions to the initial value problem for Hamilton--Jacobi equations with obstacles, 
under the assumption that the Hamiltonian is coercive.
In Section 2.2, 
we also provide a proof of optimal control representation formulas
to make our argument self-contained.

Finally,
we mention several simplifications
that can be made without loss of generality.
Under assumptions (A\ref{assump:H-peri})--(A\ref{assump:initial}),
as we see in Proposition \ref{thm:a-priori},
(\ref{OP}) admits a unique viscosity solution $u^\varepsilon$ for each $\varepsilon >0$ and
there exsits a constant $C>0$ independent of $\varepsilon$
such that 
\begin{align*}
\lVert u^\varepsilon_t \rVert_{L^\infty(\R^n\times[0,\infty))}
+\lVert Du^\varepsilon \rVert_{L^\infty(\R^n\times [0,\infty))} 
\leq C.
\end{align*}
Therefore, the value of Hamiltonian $H(y,p)$ for $|p|>C$
are irrelevant to analyze (\ref{OP}).
Modifying the value of $H(y,p)$ on $\R^n\times(\R^n\setminus B(0,R))$
for sufficiently large $R>0$ if necessary,
we may assume that there exists a constant $K_0 >0$ such that
\begin{align}\label{assump:quad-growth}
    \frac{1}{2}|p|^2-K_0 \leq H(y,p) \leq \frac{1}{2}|p|^2 + K_0 \quad \text{for all } (y,p)\in\R^n\times\R^n.
\end{align}
By this argument,
we may also assume the Lagrangian $L$ satisfies the same growth condition
\begin{align*}
    \frac{1}{2}|q|^2-K_0 \leq L(y,q) \leq \frac{1}{2}|q|^2 + K_0 \quad \text{for all } (y,q)\in\R^n\times\R^n,
\end{align*}
and thus 
optimal control representation (\ref{eq:value-fcn-intro}) remains valid under growth condition (A\ref{assump:coer}) for $H$, 
without requiring the superlinear growth condition.

\subsection*{Organization of the paper}
In Section 2, 
we summarize the standard arguments of viscosity solution theory for obstacle problems 
and introduce optimal control formulas with stopping time.
Section 3.1 is devoted to a proof of Theorem \ref{thm:main-thm}.
In Section 3.2, 
we provide an example of an obstacle problem 
whose convergence rate is $O(\varepsilon^\alpha)$ for $\alpha\in (0,1)$ but not $O(\varepsilon)$.

\section{Preliminaries}

\subsection{Viscosity solutions to obstacle problems}
In this subsection, we recall standard viscosity solution theory which we need for our obstacle problems.
All of the results presented in this subsection 
are rather standard and can be obtained by suitably adapting the basic framework of viscosity solution theory for Cauchy problems.
Nevertheless, for the sake of completeness 
and to make the argument of the obstacle problems self-contained, 
we provide the proofs.

Our main purpose in this subsection is to prove 
unique solvability and an a priori Lipschitz estimate 
of the following form:
\begin{proposition}\label{thm:a-priori}
    Assume $(\mathrm{A}\ref{assump:H-peri})$--$(\mathrm{A}\ref{assump:initial})$.
    Then, each $($\ref{OP}$)$ and $($\ref{effOP}$)$
    has a unique viscosity solution $u^\varepsilon$ and $u$,
    respectively.
    Moreover, they are Lipschitz continuous 
    and there exists a constant $C>0$ independent of $\varepsilon >0$ such that
    \begin{align*}
        \lVert u^\varepsilon_t \rVert_{L^\infty (\R^n\times[0,\infty))}
        +\lVert Du^\varepsilon\rVert_{L^\infty(\R^n\times[0,\infty))}
        \leq C.
    \end{align*}
\end{proposition}

We now recall the definition of viscosity solutions to obstacle problems.
\begin{definition}
    We say that 
    \begin{enumerate}[(a)]
        \item the function $w\in \USC (\R^n\times[0,\infty))$
        is a viscosity subsolution to (\ref{OP})
        if $w(\cdot ,0)\leq u_0$ on $\R^n$, and, 
        for any $(x_0,t_0)\in\R^n\times (0,\infty)$ and
        for any $\phi \in C^1(\R^n\times(0,\infty))$ 
        such that $w-\phi$ attains a local maximum at $(x_0,t_0)$, we have
        \begin{align}
            \max \left\{
                \phi_t(x_0,t_0)+H\left(
                    \frac{x_0}{\varepsilon},D\phi (x_0,t_0)
                \right),\ 
                w(x_0,t_0)-\psi\left(
                    x_0,\frac{x_0}{\varepsilon}
                \right)
            \right\}\leq 0;
            \label{def:visc-subsol}
        \end{align}
        \item the function $v\in \LSC (\R^n\times[0,\infty))$
        is a viscosity supersolution to (\ref{OP})
        if $v(\cdot ,0)\geq u_0$ on $\R^n$, and, 
        for any $(x_0,t_0)\in\R^n\times (0,\infty)$ and
        for any $\phi \in C^1(\R^n\times(0,\infty))$ 
        such that $v-\phi$ attains a local minimum at $(x_0,t_0)$, we have
        \begin{align}
            \max \left\{
                \phi_t(x_0,t_0)+H\left(
                    \frac{x_0}{\varepsilon},D\phi (x_0,t_0)
                \right),\ 
                v(x_0,t_0)-\psi\left(
                    x_0,\frac{x_0}{\varepsilon}
                \right)
            \right\}\geq 0;
            \label{def:visc-supsol}
        \end{align}
        \item the function $w\in C (\R^n\times[0,\infty))$
        is a viscosity solution to (\ref{OP})
        if $w$ is both a viscosity sub and supersolution to (\ref{OP}).
    \end{enumerate}
\end{definition}

\begin{remark}
    Inequality (\ref{def:visc-subsol}) is equivalent to the condition that
    \begin{align*}
        \phi_t(x_0,t_0)
        +H\left(
            \frac{x_0}{\varepsilon},D\phi (x_0,t_0)
        \right)\leq 0
        \quad \text{and}\quad 
        w(x_0,t_0)
        \leq \psi\left(
            x_0,\frac{x_0}{\varepsilon}
        \right).
    \end{align*}
    Similarly, (\ref{def:visc-supsol}) is equivalent to the condition that
    \begin{align*}
        \phi_t(x_0,t_0)
        +H\left(
            \frac{x_0}{\varepsilon},D\phi (x_0,t_0)
        \right)\geq 0
        \quad \text{or}\quad 
        v(x_0,t_0)
        \geq \psi\left(
            x_0,\frac{x_0}{\varepsilon}
        \right).
    \end{align*}
    Thus, (\ref{OP}) can be written as
    \begin{align*}
        \left\{
        \begin{aligned}
            &u^\varepsilon_t+H\left(\frac{x}{\varepsilon},Du^\varepsilon\right)\leq 0
            \quad\text{in } \R^n\times (0,\infty),\\
            &u^\varepsilon_t + H\left(\frac{x}{\varepsilon},Du^\varepsilon\right)\geq 0
            \quad\text{in } 
            \left\{ u^\varepsilon <\psi\left(x,\frac{x}{\varepsilon}\right)\right\},\\
            &u^\varepsilon\leq \psi \left(x,\frac{x}{\varepsilon}\right) \quad\text{on } \R^n\times [0,\infty),\quad 
            u^\varepsilon (x,0)=u_0(x) \quad\text{on }\R^n.
        \end{aligned}
        \right.
    \end{align*}
\end{remark}

\begin{remark}
For notational simplicity,
throughout the proofs of the propositions in this section,
we consider only the case $\varepsilon=1$ for (\ref{OP}) 
and write $\psi (x)$ for $\psi (x,x/\varepsilon)$ whenever there is no risk of confusion.
\end{remark}

We start from showing the comparison principle for obstacle problems.
To prove uniqueness, existense and solution-wise Lipschitz estimate, we only need the condition that
\begin{align}
    H\in\BUC (\R^n\times B(0,R)) \quad \text{for every } R>0.
    \label{assump:H-BUC}
\end{align}
If we assume (A\ref{assump:H-peri}) for $H$,
(\ref{assump:H-BUC}) immediately holds.

\begin{proposition}\label{thm:comparison}
    Assume $(\mathrm{A}\ref{assump:H-peri})$, 
    $(\mathrm{A}\ref{assump:psi-peri})$
    and $(\mathrm{A}\ref{assump:initial})$,
    and take $T\in(0,\infty)$ arbitrarily.
    Let $w\in\USC(\R^n\times [0,T))$ 
    and $v\in\LSC(\R^n\times [0,T))$ 
    be a bounded viscosity subsolution and supersolution to $($\ref{OP}$)$, respectively, 
    and suppose that at least one of $w$ and $v$ 
    is Lipschitz continuous on $\R^n \times [0,T)$.
    Then, we have that
    \begin{align*}
        w(x,0) \leq v(x,0)\text{ on } \R^n 
        \quad \text{implies} \quad
        w(x,t) \leq v(x,t) \text{ on } \R^n\times [0,T).
    \end{align*}
\end{proposition}

\begin{proof}
    Let $w\in\USC(\R^n\times [0,T))$ 
    and $v\in\LSC(\R^n\times [0,T))$ 
    be a bounded viscosity sub and supersolution to (\ref{OP}), respectively, with $w(\cdot,0)\leq v(\cdot,0)$.
    Without loss of generality, 
    we can assume $w\in\Lip (\R^n\times[0,T))$.
    
    For positive parameters $\delta, \theta, \lambda>0$, 
    we consider the auxiliary function $\Phi:\R^{2n}\times[0,T)^2\to\R$ defined by
    \begin{align*}
        &\Phi(x,y,t,s)=\Phi^{\delta,\theta,\lambda}(x,y,t,s)\\
        &= w(x,t)-v(y,s)-\frac{|x-y|^2+|t-s|^2}{\delta}
        -\theta (|x|^2+|y|^2)-\lambda(t+s)
        -\frac{\delta}{T-t}-\frac{\delta}{T-s}.
    \end{align*}
    We assume by contradiction that
    $M:= \sup_{\R^n\times[0,T)} (w-v) >0$.
    Then, there exists $(x_1,t_1)\in\R^n\times[0,T)$
    such that $(w-v)(x_1,t_1)\geq M/2$ and it follows that
    \begin{align}
        \Phi^{\delta_1,\theta,\lambda}(x_1,x_1,t_1,t_1)
        = w(x_1,t_1)-v(x_1,t_1) -2\theta|x_1|^2 
        - 2\lambda t_1 - \frac{2\delta_1}{T-t_1}
        > \frac{M}{4}
        \label{eq:2.CP-1}
    \end{align}
    for any sufficiently small $\delta_1,\theta,\lambda >0$.
    Here, we fix $\delta_1,\theta,\lambda$ satisfying
    $(\ref{eq:2.CP-1})$ and 
    let $(x_\delta,y_\delta,t_\delta, s_\delta)$ 
    be a maximum point of $\Phi^{\delta,\theta,\lambda}$
    for each $\delta\in (0,\delta_1)$.
    By usual computation for auxiliary functions of
    doubling variable methods, we have
    \begin{align}
        \frac{|x_\delta - y_\delta|^2 + |t_\delta -s_\delta|^2}{\delta}\leq C
        \quad \text{and} \quad
        |x_\delta| + |y_\delta| \leq \frac{C}{\sqrt{\theta}}
        \quad \text{for every } \delta\in(0,\delta_1),
        \label{eq:2.CP-2}
    \end{align}
    for a constant $C>0$ independent of $\delta,\theta,\lambda$.
    Moreover, by Lipschitz continuity of $w$ 
    and the inequality 
    $\Phi(y_\delta,y_\delta,t_\delta,s_\delta)\leq \Phi(x_\delta,y_\delta,t_\delta,s_\delta)$,
    we obtain
    \begin{align}
        \frac{|x_\delta-y_\delta|^2}{\delta}
        &\leq w(x_\delta,t_\delta)-w(y_\delta,t_\delta)
        +\theta(|y_\delta|^2 - |x_\delta|^2)\notag \\
        &\leq \lVert Dw \rVert_{L^\infty}\cdot |x_\delta-y_\delta| + \theta (|x_\delta|+|y_\delta|)|x_\delta-y_\delta|.\label{eq:2.CP-2.5}
    \end{align}
    Combining (\ref{eq:2.CP-2}) and (\ref{eq:2.CP-2.5}),
    we have 
    \begin{align}
        |x_\delta-y_\delta|\leq C \delta
        \quad \text{for every } \delta\in(0,\delta_1) .\label{eq:2.CP-3}
    \end{align}
    
    To use the definition of viscosity sub and supersolutions,
    we show that neither $t_\delta$ nor $s_\delta$
    can be $=0$ for each sufficiently small $\delta>0$.
    Since $(x_\delta,y_\delta,t_\delta,s_\delta)$ is bounded due to (\ref{eq:2.CP-2}),
    taking a subsequense if necessary,
    we can assume $(x_\delta,y_\delta,t_\delta,s_\delta)\to (x_0,y_0,t_0,s_0)$ as $\delta\to+ 0$ for some $(x_0,y_0,t_0,s_0)\in\R^{2n}\times[0,T]^2$.
    Here, if $s_0=0$,
    then 
    \begin{align*}
        0&<\frac{M}{4}\leq \Phi(x_\delta,y_\delta,t_\delta,s_\delta)
        \leq w(x_\delta,t_\delta) - w(y_\delta,s_\delta)
        +w(y_\delta,s_\delta)-v(y_\delta,s_\delta),
    \end{align*}
    and by taking $\limsup_{\delta\to +0}$, we have
    \begin{align*}
        \frac{M}{4}\leq w(y_0,0)-v(y_0,0)\leq 0.
    \end{align*}
    This is a contradiction, 
    and it follows that there exists $\mu>0$ with $s_\delta\geq \mu$ for all $\delta\in(0,\delta_1)$.
    For this $\mu$, 
    it also follows that $t_\delta \geq \mu/2$ 
    for every small $\delta$ since $|t_\delta -s_\delta|\to 0$ as $\delta\to 0$.

    We noting that $(x,t)\mapsto \Phi(x,y_\delta,t,s_\delta)$
    takes its maximum at $(x_\delta,t_\delta)\in\R^n\times (0,T)$,
    viscosity subsolution test for $w$ implies that
    \begin{align}
        \frac{2(t_\delta-s_\delta)}{\delta}+\lambda +\frac{\delta}{(T-t_\delta)^2}
        +H\left( x_\delta, \frac{2(x_\delta -y_\delta)}{\delta}+2\theta x_\delta \right)\leq 0
        \quad \text{and}\quad w(x_\delta , t_\delta)\leq \psi(x_\delta).\label{eq:2.CP-4}
    \end{align}
    We also noting that $(y,s)\mapsto -\Phi(x_\delta,y,t_\delta,s)$ takes its minimum at $(y_\delta,s_\delta)\in\R^n\times(0,T)$,
    vicosity supersolution test for $v$ implies that
    \begin{align}
        \frac{2(t_\delta-s_\delta)}{\delta}-\lambda -\frac{\delta}{(T-s_\delta)^2}
        +H\left( y_\delta, \frac{2(x_\delta -y_\delta)}{\delta}-2\theta y_\delta \right)\geq 0
        \quad \text{or}\quad v(y_\delta , s_\delta)\geq \psi(y_\delta).\label{eq:2.CP-5}
    \end{align}
    We need to consider two cases regarding to (\ref{eq:2.CP-5}).
    First, 
    suppose that for some $\theta>0$ with (\ref{eq:2.CP-1}), 
    there exists a subsequence $\{\delta_k\}_{k=1}^\infty$ with $\delta_k\to 0$
    such that $v(y_{\delta_k},t_{\delta_k})\geq \psi(y_{\delta_k})$ for all $k$. 
    Then, we have
    \begin{align*}
        v(y_{\delta_k},s_{\delta_k})\geq \psi(x_{\delta_k})+\psi(y_{\delta_k})-\psi(x_{\delta_k})
        \geq w(x_{\delta_k},t_{\delta_k}) - \omega_\psi (|x_{\delta_k}-y_{\delta_k}|),
    \end{align*}
    where $\omega_\psi$ is a modulus of continuity of $\psi$
    on a ball $B(0,C/\sqrt{\theta})$ determined by (\ref{eq:2.CP-2}).
    Thus, we have
    \begin{align*}
        \frac{M}{4}\leq \limsup_{k\to\infty}\Phi(x_{\delta_k},y_{\delta_k},t_{\delta_k},s_{\delta_k})
        \leq \lim_{k\to\infty}\omega_\psi (|x_{\delta_k}-y_{\delta_k}|) =0,
    \end{align*}
    which is a contradiction.
    Otherwise, if for every $\theta$ with (\ref{eq:2.CP-1})$, v(y_\delta,s_\delta)< \psi(y_\delta)$ holds
    for any sufficiently small $\delta$,
    by taking difference between the first inequalities 
    of (\ref{eq:2.CP-4}) and (\ref{eq:2.CP-5}),
    we have 
    \begin{align*}
        0< 2\lambda 
        < 2\lambda + \frac{\delta}{(T-t_\delta)^2}+\frac{\delta}{(T-s_\delta)^2}
        \leq \omega_H (|x_\delta - y_\delta|+2\theta(|x_\delta - y_\delta|)).
    \end{align*}
    Here, $\omega_H$ is a modulus of continuity of $H$ 
    on $\R^n\times B(0,R)$ for some $R>0$ independent of $\theta$ and $\delta$,
    and we can obtain the existence of such an $\omega_H$ by (\ref{eq:2.CP-2}), (\ref{eq:2.CP-3}) and condition (\ref{assump:H-BUC}).
    Therefore, we obtain
    \begin{align*}
        2\lambda 
        \leq \limsup_{\theta\to +0}
        \limsup_{\delta \to +0} 
        \omega_H (
            |x_\delta-y_\delta|
            +2\theta(|x_\delta|+|y_\delta|)
        )
        \leq \limsup_{\theta\to +0}
        \omega_H(2C\sqrt{\theta})
        =0,
    \end{align*}
    which is also a contradiction.
\end{proof}

Next, we provide a stability result and Perron's method for our obstacle problems.

\begin{proposition}\label{thm:stability}
    Assume $(\mathrm{A}\ref{assump:H-peri})$,
    $(\mathrm{A}\ref{assump:psi-peri})$
    and $(\mathrm{A}\ref{assump:initial})$.
    Let $\{w^i\}_{i\in I}\subset \USC (\R^n\times[0,\infty))$ be a family of viscosity subsolutions to $($\ref{OP}$)$ and set $w(x,t)=\sup_{i\in I} w^i(x,t)$. 
    Then, $w^*$ is a viscosity subsolution to $($\ref{OP}$)$ if $w^*$ is finite at each point.
\end{proposition}

Here, 
for a function $f:\Omega(\subset\R^n\times [0,\infty))\to [-\infty,+\infty]$, 
$f^*$ and $f_*$ denotes upper and lower semicontinuous envelope of $f$, that is, the function defined by
\begin{align*}
    f^*(z):=\lim_{r\to+0} \sup_{
        \substack{
            \widetilde{z}\in\Omega\\
            |\widetilde{z}-z|\leq r
        }
    } f(\widetilde{z})
    \quad \text{and} \quad
    f^*(z):=\lim_{r\to+0} \inf_{
        \substack{
            \widetilde{z}\in\Omega\\
            |\widetilde{z}-z|\leq r
        }
    } f(\widetilde{z}),
    \quad \text{respectively.}
\end{align*}

\begin{proof}
    Assume $w^*-\phi$ takes its strict maximum at $(x_0,t_0)\in \R^n\times (0,\infty)$ for $\phi\in C^1(\R^n\times (0,\infty))$.
    By the standard argument for upper semicontinuous functions, we can find sequences 
    $\{w^{i_k}\}_{k=1}^\infty$ 
    and $\{(x_k,t_k)\}_{k=1}^\infty$
    satisfying
    \begin{align*}
        \left\{
        \begin{aligned}
        &w^{i_k}-\phi \text{ takes its local maximum at } (x_k,t_k),\\
        &(x_k,t_k)\to (x_0,t_0) \quad
        \text{and}\quad w^{i_k}(x_0,t_0)\uparrow w(x_0,t_0)
        \quad \text{as}\quad k\to\infty.
        \end{aligned}
        \right.
    \end{align*}
    Then, we have
    \begin{align*}
        \phi_t(x_0,t_0) + H(x_0,D\phi(x_0,t_0))
        =\lim_{k\to\infty} \{\phi_t(x_k,y_k)+ H(x_k,\phi(x_k,t_k))\}\leq 0
    \end{align*}
    and
    \begin{align*}
        w(x_0,t_0)=\lim_{k\to\infty}w^{i_k}(x_0,t_0)
        \leq \limsup_{k\to\infty}w^{i_k}(x_k,t_k)
        \leq \lim_{k\to\infty}\psi (x_k) =\psi(x_0).
    \end{align*}
    In particular, 
    since $\psi$ is continuous 
    and $w^*$ is the smallest upper semicontinuous majorant of $w$,
    the inequality $w^*(x_0,t_0)\leq \psi (x_0)$ holds.
\end{proof}

\begin{proposition}\label{thm:Perron}
    Assume $(\mathrm{A}\ref{assump:H-peri})$, $(\mathrm{A}\ref{assump:coer})$, 
    $(\mathrm{A}\ref{assump:psi-peri})$ and $(\mathrm{A}\ref{assump:initial})$.
    Set
    \begin{align*}
        C_0 := \sup\left\{
            |H(y,p)|
            \;:\; y\in\R^n,\ 
            |p|\leq \lVert Du_0\rVert_{L^\infty(\R^n)}
        \right\},
    \end{align*}
    and
    \begin{align*}
        \phi_-(x,t) := u_0(x)-C_0t,
        \quad 
        \phi_+(x,t) := (u_0(x) +C_0t) \wedge \psi(x),
    \end{align*}
    for $(x,t)\in\R^n\times [0,\infty)$.
    Let $\mathcal{S}$ be the class of
    viscosity subsolutions to $($\ref{OP}$)$
    determined by $w\in\mathcal{S}$ if and only if
    \begin{align*}
        w\in \USC(\R^n\times[0,\infty)),\ w
        \text{ is a viscosity subsolution to }
        ( \text{\ref{OP}} ),
        \text{ and }
        \phi_- \leq w \leq \phi_+,
    \end{align*}
    and set $v(x,t)=\sup_{w\in\mathcal{S}}w(x,t)$.
    Then, $v^*\in\USC(\R^n\times[0,\infty))$ and $v_*\in\LSC(\R^n\times[0,\infty))$ is a viscosity subsolution and supersolution to $($\ref{OP}$)$,
    respectively.
\end{proposition}

\begin{proof}
    By Proposition \ref{thm:stability},
    $v^*$ is a viscosity subsolution to (\ref{OP}).
    To show that $v_*$ is a viscosity supersolution,
    assume $v_*-\phi$ takes 
    its strict minimum $(v_*-\phi)(x_0,t_0)=0$ at $(x_0,t_0)$
    for $\phi\in C^1(\R^n\times(0,\infty))$.
    
    If $v_*(x_0,t_0)=\phi_+(x_0,t_0)$,
    we have 
    \begin{align*}
        \max\{
            \phi_t(x_0,t_0)+H(x_0,D\phi(x_0,t_0)),\ 
            v_*(x_0,t_0)-\psi(x_0)
        \}\geq 0
    \end{align*}
    since $\phi_+$ is a viscosity supersolution 
    and $\phi_+-\phi$ also takes its minimum at $(x_0,t_0)$.
    
    Otherwise, if $v_*(x_0,t_0)<\phi_+(x_0,t_0)$,
    since $v_*(x_0,t_0)<\psi(x_0)$ by the construction of $\phi_+$,
    we need to show that $\phi_t (x_0,t_0) + H(x_0,D\phi(x_0,t_0))\geq 0$.
    Assume, by contradiction,
    that
    $\phi_t(x_0,t_0) + H(x_0,D\phi(x_0,t_0))<0$.
    We can take $r,\delta>0$ such that
    \begin{align*}
        \left\{
        \begin{aligned}
            &\phi_t+ H(x,D\phi) <-\delta
            && \text{in } B(x_0,r)\times (t_0 -r,t_0 +r),\\
            &v^*-\phi \ (\geq v_*-\phi) >\delta 
            &&\text{on } \partial(B(x_0,r)\times (t_0-r , t_0+r)),\\
            &\phi\ (\leq v_*\leq v^*)\leq \phi_+-\delta
            &&\text{in } B(x_0,r)\times (t_0-r,t_0+r),
        \end{aligned}
        \right.
    \end{align*}
    and we set 
    \begin{align*}
        \widetilde{v}(x,t)=\left\{
        \begin{aligned}
            & v^*(x,t) \vee (\phi(x,t)+\delta) 
            && \text{if }(x,t)\in B(x_0,r)\times (t_0-r,t_0+r),\\
            & v^*(x,t)
            &&\text{if } (x,t) \not\in B(x_0,r)\times (t_0-r,t_0+r).
        \end{aligned}
        \right.
    \end{align*}
    Then, by Proposition \ref{thm:stability} 
    and the inequality
    $\phi + \delta \leq \phi_+ \leq\psi(x)$,
    we obtain that
    $\widetilde{v}$ is a viscosity subsolution of (\ref{OP})
    with $\phi_-\leq\widetilde{v}\leq \phi_+$.
    This is a contradiction since $\widetilde{v}(x_0,t_0)=\phi(x_0,t_0)+\delta>v(x_0,t_0)$.
\end{proof}

\begin{proof}[Proof of Proposition \ref{thm:a-priori}]
Let $\mathcal{S}$ and $v$ be as in Proposition \ref{thm:Perron}.
First, we show that $v^*$ is Lipschitz continuous.

Since $H$ is coercive, 
there exists a constant $C_1>0$ such that 
$H(y,p)\geq -C_1$ for all $(y,p)$.
This implies 
\begin{align}
    (v^*)_t - C_1 \leq (v^*)_t+H(x,Dv^*)\leq 0
    \quad\text{in } \R^n\times (0,\infty),\label{eq:apriori-1}
\end{align}
in the sense of viscosity solutions.
To obtain the estimate of $(v^*)_t$ from below,
we consider the function
\begin{align*}
        \phi_-^s (x,t):=\left\{
        \begin{aligned}
            &u_0(x)-C_0t &&\text{if } t\leq s,\\
            &v^*(x,t-s) - C_0 s &&\text{if } t>s,
        \end{aligned}
        \right.
\end{align*}
for each $s>0$.
We claim:
\begin{customdefinition}{Claim}
    For any $s>0$, it follows that $\phi_-^s\in\mathcal{S}$.
\end{customdefinition}
\begin{subproof}[Proof of Claim.]
    For the subsolution test for $\phi_-^s$,
    assume that $\phi_-^s - \phi $ takes maximum at $(x_0,t_0)$ for $\phi\in C^1 (\R^n\times (0,\infty))$.
    If $t_0\neq s$, we directly have
    \begin{align*}
    \max\{\phi_t(x_0,t_0)+H(x_0,D\phi(x_0,t_0)),\ \phi^s_-(x_0,t_0)-\psi(x_0)\}\leq 0.
    \end{align*}
    Otherwise, if $t_0=s$, then for each small $h>0$,
    we have
    \begin{align}
        \phi(x_0,s) - \phi(x_0,s-h)
        &\leq \phi^s_-(x_0,s)-\phi^s_-(x_0,s-h)
        = u_0(x_0)-C_0s - u_0(x_0)+C_0(s-h)\notag \\ 
        &= - C_0h.\label{eq:apriori-2}
    \end{align}
    Dividing both side by $h$ and taking $\lim_{h\to +0}$,
    we have $\phi_t(x_0,s)\leq -C_0$.
    Also, since $u_0$ is Lipschitz in space,
    we have $|D\phi(x_0,s)|\leq \lVert D \phi^s_-(\cdot, s)\rVert_{L^\infty(\R^n)}=\lVert Du_0 \rVert_{L^\infty(\R^n)}$.
    Thus, it follows that
    \begin{align*}
        \phi_t(x_0,s)+H(x_0,D\phi(x_0,s))\leq -C_0 + C_0 =0.
    \end{align*}
    The rest inequality $\phi^s_-(x_0,s) -\psi(x_0)\leq 0$ holds by its definition.
    Next, we comfirm the inequality $\phi_-\leq \phi_-^s \leq \phi_+$.
    By the definition of $v$ and the fact 
    that $v^*$ is the smallest upper semicontiuous
    majorant of $v$, 
    we have $\phi_- \leq v^* \leq \phi_+$ on $\R^n\times[0,\infty)$.
    Therefore, it holds, for any $(x,t)$ with $t>s$, that
    \begin{align*}
        &\phi^s_-(x,t)= v^*(x,t-s)-C_0s \geq \phi_-(x,t-s)-C_0 s = \phi_-(x_0,t),\\
        &\phi^s_-(x,t)\leq \phi_+(x,t-s)-C_0s
        \leq\phi_+(x,t).
    \end{align*}
    For a point $(x,t)$ with $0\leq t\leq s$, there is nothing to prove $\phi_-\leq \phi_-^s\leq \phi_+$.
\end{subproof}
From this Claim,
we have $\phi^s_- \leq v^*$ by the definition of $v$,
and this implies 
\begin{align*}
    v^*(x,t)-v^*(x,t-s)\geq -C_0 s \quad \text{for every } (x,t) \text{ with } 0\leq s < t.
\end{align*}
One can see that this inequality implies
\begin{align}\label{eq:apriori-3}
    -(v^*)_t \leq C_0 \quad \text{in } \R^n\times (0,\infty)
\end{align}
in the sense of viscosity solutions.

By the coercivity of $H$,
we note that it follows,
for all $M>0$,
that there exists $R>0$ such that 
\begin{align}\label{eq:apriori-4}
    \text{if } p\in\R^n \text{ satisfies }\inf_{y\in\R^n}H(y,p) \leq M, \quad \text{then} \quad |p|\leq R.
\end{align}
By (\ref{eq:apriori-1}) and (\ref{eq:apriori-3}),
we obtain
\begin{align}
    H(x,Dv^*)\leq \max\{C_0, C_1\}
    \label{eq:apriori-4.5}
\end{align}
in the sense of viscosity solutions. 
From (\ref{eq:apriori-1}),
(\ref{eq:apriori-3}), (\ref{eq:apriori-4})
and (\ref{eq:apriori-4.5}),
$v^*$ satisfies, for a constant $C>0$,
\begin{align}\label{eq:apriori-5}
    |(v^*)_t|+|Dv^*|\leq C \quad \text{in } \R^n\times (0,\infty)
\end{align}
in the sense of viscosity solutions.
Inequality (\ref{eq:apriori-5}) implies that the subsolution $v^*$ is Lipschitz continuous
in space and time with $C$ being a Lipschitz bound 
(see \cite[Proposition 1.14]{MR3135341}),
and by Proposition \ref{thm:comparison},
we have $v^*\leq v_*$. 
Therefore, $v= v^* = v_*$ is a unique viscosity solution to (\ref{OP}).

If we consider (\ref{OP}) for general $\varepsilon>0$,
by the determination of $C>0$ in (\ref{eq:apriori-5}),
one can see that this Lipschitz bound is independent of $\varepsilon$.
\end{proof}

\subsection{Representation formulas}
We present the optimal control framework which describes viscosity solutions to (\ref{OP}) and (\ref{effOP}).
In this subsection,
we assume (\ref{assump:quad-growth}) instead of (A\ref{assump:coer}) for Hamiltonian $H$.
We denote by $L=L(y,q)$ and $\overline{L}=\overline{L}(q)$,
the Legendre transform of $H$ and $\overline{H}$, respectively, that is,
\begin{align}
    L(y,q):=\sup_{p\in\R^n} \{p\cdot q - H(y,p)\},
    \quad \overline{L}(q)
    :=\sup_{p\in\R^n}\{p\cdot q-\overline{H}(p)\}.
\end{align}
If $H$ satisfies (\ref{assump:H-BUC}), (\ref{assump:quad-growth}) and (A\ref{assump:convex}),
then $L$ also satisfies the same conditions:
\begin{gather}
    L\in \BUC(\R^n\times B(0,R)) \quad \text{for all } R>0, \label{assump:L-BUC}\\
    \frac{1}{2}|q|^2-K_0 \leq L(y,q) \leq \frac{1}{2}|q|^2+K_0
    \quad \text{for all } (y,q)\in\R^n\times\R^n,
    \label{assump:L-quad-growth}
\end{gather}
and
\begin{gather}
    \text{the map }q \mapsto L(y,q) \text{ is convex for each } y\in\R^n.\label{assump:L-convex}
\end{gather}
Assume $u_0, \psi\in C(\R^n)$ satisfies, for a constant $c\in\R$,
\begin{gather}
    c< u_0(x)\leq \psi(x) \quad \text{for all } x\in\R^n.
    \label{assump:termial-stopping-costs}
\end{gather}
For each $(x,t)\in \R^n\times [0,\infty)$,
we define the value function $v$ by
\begin{align}
    v(x,t) = \inf_{\substack{\gamma\in\mathcal{A}(x,t)\\\theta\in[0,t]}}
    J(x,t,\gamma,\theta)\label{eq:value-fcn}
\end{align}
where
\begin{equation*}
    \A(x,t):=\{\gamma\in \AC([0,t];\R^n)
\; : \; \gamma(t)=x\}
\end{equation*}
is the set of all admissible trajectories,
and $J$ is a cost functional defined by
\begin{align*}
    J(x,t,\gamma,\theta)
    :=\int_\theta^t L(\gamma(s),\dot{\gamma}(s))ds
    +\mathbf{1}_{\{\theta=0\}}u_0(\gamma(0))
    +\mathbf{1}_{\{\theta>0\}}\psi(\gamma(\theta)).
\end{align*}

In this subsection, 
we review that the value function $v$ solves a corresponding obstacle problem of a Hamilton--Jacobi equation.
For the basic properties between value functions for finite-horizon optimal control problems and the corresponding Hamilton--Jacobi equations, we refer to \cite{MR4328923}. 
The relationship between value functions for infinite-horizon problems and stationary Hamilton--Jacobi equations with obstacles is provided in \cite{MR1484411}.
Drawing on these literatures,
we present a proof of the optimal control representation formulas for (\ref{OP}) and (\ref{effOP}).

First, we prove the dynamic programming principle that the value function $v$ satisfies.

\begin{proposition}[Dynamic programming principle]\label{thm:DPP}
    Assume $(\ref{assump:L-BUC})$--$(\ref{assump:termial-stopping-costs})$,
    and let $v$ be the function defined by $(\ref{eq:value-fcn})$.
    Then,
    $v$ satisfies
    \begin{align}
        v(x,t)
        =\inf_{\substack{\gamma\in\mathcal{A}(x,t)\\\theta\in[0,t]}}\left\{
        \int_{\theta\vee s}^t L(\gamma(r),\dot{\gamma}(r))dr
        +\mathbf{1}_{\{\theta\leq s\}}v(\gamma(s),s)
        +\mathbf{1}_{\{\theta> s\}}\psi (\gamma(\theta))
        \right\}\label{eq:DPP}
    \end{align}
     for all $(x,t)\in\R^n\times[0,\infty)$ and $s\in[0,t]$.
\end{proposition}

\begin{proof}
    Fix $(x,t)$ and $s\in[0,t]$ arbitrarily.
    Then, for all $\gamma\in\mathcal{A}(x,t)$ and $\theta\in[0,t]$ with $s<\theta\leq t$,
    we have
    \begin{align}
        J(x,t,\gamma,\theta)
        = \int_\theta ^t L(\gamma(r),\dot{\gamma}(r))dr
        + \psi (\gamma(\theta))
        \geq (\text{Right-hand side of (\ref{eq:DPP})}).\label{eq:DPP-1}
    \end{align}
    And for all $\gamma\in\mathcal{A}(x,t)$ and $\theta\in[0,t]$ with $0\leq \theta\leq s$,
    we have 
    \begin{align}
        J(x,t,\gamma,\theta) 
        &=\int_s^t L(\gamma(r),\dot{\gamma}(r))dr
        +J(\gamma(s),s,\gamma|_{[0,s]},\theta)\notag\\
        &\geq \int_s^t L(\gamma(r),\dot{\gamma}(r))dr
        +v(\gamma(s),s)
        \geq (\text{Right-hand side of (\ref{eq:DPP})}).\label{eq:DPP-2}
    \end{align}
    We combining (\ref{eq:DPP-1}) and (\ref{eq:DPP-2}),
    and taking $\inf_{\gamma,\theta}$, it follows that
    \begin{align*}
        v(x,t)= \inf_{\substack{\gamma\in\mathcal{A}(x,t)\\ \theta\in[0,t]}} J(x,t,\gamma,\theta)\geq (\text{Right-hand side of (\ref{eq:DPP})}).
    \end{align*}
    Conversely, 
    for all $\gamma\in\mathcal{A}(x,t)$ and $\theta\in[0,t]$ with $s<\theta\leq t$,
    we have
    \begin{align}
        v(x,t) \leq J(x,t,\gamma,\theta)
        =\int_{\theta\vee s}^t L(\gamma(r),\dot{\gamma}(r))dr + 
        \mathbf{1}_{\{\theta\leq s\}} v(\gamma(s),s)
        +\mathbf{1}_{\{\theta> s\}}\psi(\gamma(\theta)).
        \label{eq:DPP-3}
    \end{align}
    For $\gamma\in\mathcal{A}(x,t)$ and $\theta\in[0,t]$ with $0\leq \theta\leq s$, 
    and for $\widetilde{\gamma}\in\mathcal{A}(\gamma(s),s)$
    and $\widetilde{\theta}\in[0,s]$,
    we let 
    \begin{align*}
        \xi (r) = \left\{
        \begin{aligned}
            &\widetilde{\gamma}(r) &&\text{if } 0\leq t\leq s,\\
            &\gamma (r) &&\text{if } s< r\leq t.
        \end{aligned}
        \right.
    \end{align*}
    Then, we have $\xi\in\mathcal{A}(x,t)$ and
    \begin{align}
        v(x,t) \leq J(x,t,\xi,\widetilde{\theta})
        =\int_s^t L(\gamma(r),\dot{\gamma}(r))dr 
        + J(\gamma(s),s,\widetilde{\gamma},\widetilde{\theta}).\label{eq:DPP-4}
    \end{align}
    Thus, by taking $\inf_{\widetilde{\gamma},\widetilde{\theta}}$ in (\ref{eq:DPP-4}), 
    we have, for all $\gamma\in\mathcal{A}(x,t)$ and $\theta\in[0,t]$ with $0\leq\theta\leq s$,
    \begin{align}
        v(x,t)&\leq \int_s^t L(\gamma(r),\dot{\gamma}(r))dr
        +v(\gamma(s),s)\notag\\
        &=\int_{\theta\vee s}^t L(\gamma(r),\dot{\gamma}(r))dr + 
        \mathbf{1}_{\{\theta\leq s\}} v(\gamma(s),s)
        +\mathbf{1}_{\{\theta> s\}}\psi(\gamma(\theta)).\label{eq:DPP-5}
    \end{align}
    Again, 
    we combining (\ref{eq:DPP-3}) and (\ref{eq:DPP-5}),
    and taking $\inf_{\gamma,\theta}$,
    it follows that 
    \begin{align*}
        v(x,t)\leq (\text{Right-hand side of (\ref{eq:DPP})}).
    \end{align*}
\end{proof}

Furthermore, as shown in the propositions below, we obtain additional information on the behavior of the stopping time.

\begin{proposition}\label{thm:ex-DPP}
    Under the same assumptions of Proposition $\ref{thm:DPP}$,
    if $v(x,t)<\psi(x)$, 
    there exists $s_0=s_0(x,t)\in [0,t)$ such that
    \begin{align*}
        v(x,t)
        =\inf_{\gamma\in\mathcal{A}(x,t)}\left\{
        \int_s^t L(\gamma(r),\dot{\gamma}(r))dr
        + v(\gamma(s),s)
        \right\}
    \end{align*}
    for all $s\in(s_0,t]$.
\end{proposition}

\begin{proof}
    Fix $(x,t)\in\R^n\times[0,\infty)$ satisfying
    $v(x,t)<\psi(x)$ arbitrarily.
    For $\delta >0$, let 
    \begin{align*}
        \mathbb{A}^\delta (x,t)
        :=\{(\gamma,\theta)\in\mathcal{A}(x,t)\times [0,t]\; : \; J(x,t,\gamma,\theta)< v(x,t)+\delta\}.
    \end{align*}
    Then, by the same computations to the proof of Proposition \ref{thm:DPP},
    one can show that 
    \begin{align*}
        v(x,t)
        &=\inf_{(\gamma,\theta)\in \mathbb{A}^\delta(x,t)}
        J(x,t,\gamma,\theta)\\
        &=\inf_{(\gamma,\theta)\in\mathbb{A}^\delta(x,t)}
        \left\{
        \int_{\theta\vee s}^t L(\gamma(r),\dot{\gamma}(r))dr
        +\mathbf{1}_{\{\theta\leq s\}}v(\gamma(s),s)
        +\mathbf{1}_{\{\theta> s\}}\psi (\gamma(\theta))
        \right\}
    \end{align*}
    for each $s\in[0,t]$ and $\delta >0$.
    
    We show that there exists $\delta_0>0$ and $s_0\in[0,t)$
    such that, for all $\delta\in(0,\delta_0]$,
    \begin{align}
    (\gamma,\theta)\in\mathbb{A}^\delta (x,t)
    \quad\text{implies}\quad \theta\leq s_0\ (<t).
    \label{eq:ex-DPP-1}
    \end{align}
    Assume by contradiction, 
    there exist $\{\delta_k\}_{k=1}^\infty$ and 
    $\{(\gamma_k,\theta_k)\}_{k=1}^\infty$ satisfying
    \begin{align*}
        \delta_k\downarrow 0,\quad
        (\gamma_k,\theta_k)\in \mathbb{A}^{\delta_k} (x,t)\quad
        \text{and} \quad \theta_k\uparrow t \quad 
        \text{as } k\to\infty.
    \end{align*}
    Then, by the growth condition (\ref{assump:L-quad-growth}) of $L$ implies
    \begin{align*}
        |x-\gamma_k(\theta_k)|
        &\leq \int_{\theta_k}^t |\dot{\gamma_k}(r)|dr
        \leq (t-\theta_k)^{1/2} \left(
        \int_{\theta_k}^t |\dot{\gamma_k}(r)|^2dr\right)^{1/2}\\
        &\leq (t-\theta_k)^{1/2}\left(
        \int_{\theta_k}^t (2L(\gamma_k(r),\dot{\gamma_k}(r)) +K_0)dr \right)^{1/2}\\
        &\leq (t-\theta_k)^{1/2}
        \left(
        2\{v(x,t) +\delta_k -\inf_{z\in\R^n}\psi(z)\}+K_0(t-\theta_k)
        \right)^{1/2}
        \leq C(t-\theta_k)^{1/2}.
    \end{align*}
    Therefore, we have
    \begin{align}
        v(x,t)=\lim_{k\to\infty} J(x,t,\gamma_k,\theta_k)\geq \lim_{k\to\infty}
        \left\{\int_{\theta_k}^t (-C_1)dr+\psi(\gamma_k(\theta_k))\right\}
        =\psi(x),\label{eq:ex-DPP-2}
    \end{align}
    where $C_1>0$ is a constant,
    determined
    by (\ref{assump:L-BUC}) and (\ref{assump:L-quad-growth}),
    such that $L(y,q)\geq-C_1$ for all $(y,q)\in\R^n\times\R^n$.
    Inequality (\ref{eq:ex-DPP-2}) contradicts to
    the assumption that $v(x,t)<\psi(x)$.

    From (\ref{eq:ex-DPP-1}),
    we have,
    for all $\delta\in(0,\delta_0]$ and $s\in(s_0,t]$,
    \begin{align*}
        v(x,t)
        &= \inf_{(\gamma,\theta)\in\mathbb{A}^\delta (x,t)}
        \left\{
        \int_{\theta\vee s}^t L(\gamma(r),\dot{\gamma}(r))dr+\mathbf{1}_{\{\theta\leq s\}}v(\gamma(s),s)
        +\mathbf{1}_{\{\theta> s\}}\psi (\gamma(\theta))
        \right\}\\
        &=\inf_{(\gamma,\theta)\in\mathbb{A}^\delta (x,t)}
        \left\{
        \int_{s}^t L(\gamma(r),\dot{\gamma}(r))dr+v(\gamma(s),s)
        \right\}\\
        &=\inf_{\gamma\in\mathcal{A}(x,t)}
        \left\{\int_s^t L(\gamma(r),\dot{\gamma}(r))dr +v(\gamma(s),s)\right\}.
    \end{align*}
\end{proof}

\begin{proposition}\label{thm:lsc-in-theta}
    Under the same assumptions of Proposition $\ref{thm:DPP}$,
    for each $(x,t)$,
    the function $[0,t]\ni \theta \mapsto \inf_{\gamma\in \mathcal{A}(x,t)} J(x,t,\gamma,\theta)$
    is lower semicontinuous.
\end{proposition}

\begin{proof}
    Define $V(x,t,\theta):=\inf_{\gamma\in\mathcal{A}(x,t)} J(x,t,\gamma,\theta)$ and
    take $(x,t)\in \R^n\times (0,\infty)$ arbitrarily.
    Then, it is obvious that $\theta\mapsto V(x,t,\theta)$ is continuous on $(0,t]$ since $V(x,t,\theta)$ can be written as
    \begin{align*}
        V(x,t,\theta)
        =\inf_{\gamma\in\mathcal{A}(x,t-\theta)}
        \left\{
        \int_0^{t-\theta} L(\gamma(s),\dot{\gamma}(s))ds
        +\psi(\gamma(0))
        \right\}
        = w(x,t-\theta)
    \end{align*}
    for each $\theta\in(0,t]$, where $w=w(x,s)$ is 
    another value function with the dynamic programming principle
    \begin{align*}
        w(x,s) = \inf_{\xi \in\mathcal{A}(x,s)}
        \left\{
        \int_r^s L(\xi (\tau),\dot{\xi}(\tau))d\tau +w(\xi(r),r)
        \right\},
        \quad x\in\R^n,\ 0\leq r\leq s.
    \end{align*}
    We prove lower semicontinuity of $\theta\mapsto V(x,t,\theta)$ at $\theta=0$.
    Let $\{\theta_k\}\subset (0,t]$ be a sequence such that
    $\theta_k\downarrow 0$ as $k\to\infty$.
    Take $\delta>0$ arbitrarily small, 
    then there exists a sequence $\{\gamma_k\}\subset \mathcal{A}(x,t)$
    such that 
    \begin{align*}
        V(x,t,\theta_k)\leq J(x,t,\gamma_k,\theta_k)
        < V(x,t,\theta_k)+\delta,
        \quad \text{for all } k\in \N.
    \end{align*}
    For each $k\in\N$, 
    we define the trajectory $\widetilde{\gamma}_k\in \mathcal{A}(x,t)$ by
    \begin{align*}
        \widetilde{\gamma}_k(s)
        :=\left\{
        \begin{aligned}
            &\gamma_k (\theta_k) &&\text{if }s\in[0,\theta_k),\\
            &\gamma_k(s) &&\text{if } s\in[\theta_k,t].
        \end{aligned}
        \right.
    \end{align*}
    Then, by 
    (\ref{assump:L-BUC}) and
    (\ref{assump:termial-stopping-costs}), 
    it follows that
    \begin{align*}
        V(x,t,0)&\leq J(x,t,\widetilde{\gamma}_k,0)\\
        &=\int_0^{\theta_k} L(\gamma_k(\theta_k),0)ds
            +J(x,t,\gamma_k,\theta_k)
            +\left(
                u_0(\gamma_k(\theta_k))-\psi(\gamma_k(\theta_k))
            \right)\\
        &\leq \theta_k\cdot \sup_{z\in\R^n}|L(z,0)|
            +\left( 
                V(x,t,\theta_k)+\delta
            \right)+0.
    \end{align*}
    Therefore, we have $\liminf_{k\to\infty}V(x,t,\theta_k)\geq V(x,t,0)-\delta$, 
    and $\delta$ and $\{\theta_k\}$ can be arbitrarily chosen, it follows that 
    $\liminf_{\theta\to + 0} V(x,t,\theta)\geq V(x,t,0)$.
\end{proof}

Based on the arguments above, we prove that the value function $v$ satisfies the corresponding variational inequality of obstacle type.

\begin{proposition}\label{thm:Bellman-eq}
    Assume $(\ref{assump:L-BUC})$--$(\ref{assump:termial-stopping-costs})$.
    Then, the value function $v$ defined by $(\ref{eq:value-fcn})$ is a viscosity solution of
    \begin{align}
        \left\{
        \begin{aligned}
            &\max\{v_t+H(x,Dv),\ v-\psi(x)\}=0
            &&\text{in }\R^n\times(0,\infty),\\
            &v(x,0)=u_0(x)
            &&\text{on }\R^n,
        \end{aligned}
        \right.\label{eq:Bellman-eq}
    \end{align}
    where $H(y,p)=\sup_{q\in\R^n}\{p\cdot q - L(y,q)\}$.
\end{proposition}

\begin{proof}
    We start from showing 
    that $v$ is a viscosity subsolution to (\ref{eq:Bellman-eq}).
    Assume that $v-\phi$ takes its maximum at 
    $(x_0,t_0)\in\R^n\times(0,\infty)$
    for $\phi\in C^1(\R^n\times (0,\infty))$.
    By the definition of $v(x,t)$,
    it directory follows that
    \begin{align}
        v(x_0,t_0)\leq J(x_0,t_0,x_0(\cdot),t_0) = \psi(x_0),\label{eq:Bellman-eq-1}
    \end{align}
    where $x_0(\cdot)\in\mathcal{A}(x_0,t_0)$ is a constant path $x_0(s)\equiv x_0$.
    To prove the inequality
    $\phi_t(x_0,t_0) + H(x_0,D\phi(x_0,t_0))\leq0$,
    take $\gamma\in\mathcal{A}(x_0,t_0)\cap C^1([0,t_0];\R^n)$ and $s\in[0,t_0)$ arbitrarily.
    Since $(v-\phi)(x_0,t_0)=\max(v-\phi)$,
    we have
    \begin{align}
        v(x_0,t_0)-v(\gamma(s),s)
        &\geq \phi(x_0,t_0)-\phi(\gamma(s),s)\notag\\
        &=\int_s^{t_0} \left\{
            \phi_t(\gamma(r),\dot{\gamma}(r))+D\phi(\gamma(r),r)\cdot \dot{\gamma}(r)
        \right\}dr.
        \label{eq:Bellman-eq-2}
    \end{align}
    Proposition \ref{thm:DPP} implies
    \begin{align}
        v(x_0,t_0)
        \leq \int_s^{t_0} L(\gamma(r),\dot{\gamma}(r))dr 
        + v(\gamma(s),s).
        \label{eq:Bellman-eq-3}
    \end{align}
    Combining (\ref{eq:Bellman-eq-2}) and (\ref{eq:Bellman-eq-3}), and dividing by $t_0-s$, we have
    \begin{align*}
        0\geq \frac{1}{t_0-s}\int_s^{t_0} 
        \{\phi_t(\gamma(r),r) + D\phi(\gamma(r),r)\cdot\dot{\gamma}(r)-L(\gamma(r),\dot{\gamma}(r))\}dr.
    \end{align*}
    Then, we letting $s\uparrow t_0$,
    it follows that 
    \begin{align*}
        0\geq \phi_t(x_0,t_0)+D\phi(x_0,t_0)\cdot\dot{\gamma}(t_0)-L(x_0,\dot{\gamma}(t_0)).
    \end{align*}
    Since this inequality holds for all $C^1$ admissible path $\gamma$, we can conclude that
    \begin{align}
        0\geq 
        \phi_t(x_0,t_0)+
        \sup_{q\in\R^n}\{D\phi(x_0,t_0)\cdot q -L(x_0,q)\}
        =\phi_t(x_0,t_0)+H(x_0,D\phi(x_0,t_0)).
        \label{eq:Bellman-eq-4}
    \end{align}
    Hence, by (\ref{eq:Bellman-eq-1}) and (\ref{eq:Bellman-eq-4}), $v$ satisfies
    \begin{align*}
        \max\{v_t+H(x,Dv),\ v-\psi(x)\}\leq 0
    \end{align*}
    in the sense of viscosity solutions.

    Next, we show that $v$ is a viscosity supersolution to (\ref{eq:Bellman-eq}).
    Assume that $v-\phi$ takes its minimum at $(x_0,t_0)\in\R^n\times (0,\infty)$
    for $\phi\in C^1(\R^n\times(0,\infty))$.
    If $v(x_0,t_0)= \psi(x_0)$, there is nothing to prove.
    Suppose $v(x_0,t_0)<\psi(x_0)$.
    By Proposition \ref{thm:ex-DPP},
    we can take $s_0\in [0,t_0)$ such that
    \begin{align*}
        v(x_0,t_0)
        =\inf_{\gamma\in\mathcal{A}(x_0,t_0)}
        \left\{
        \int_{s_0}^{t_0} L(\gamma(r),\dot{\gamma}(r))dr
        +v(\gamma(s_0),s_0)
        \right\}.
    \end{align*}
    Let
    \begin{align*}
        \widetilde{\mathcal{A}}(x_0,t_0)
        :=\left\{\gamma\in\mathcal{A}(x_0,t_0)
        \; : \; 
        \int_{s_0}^{t_0}L(\gamma(r),\dot{\gamma}(r))dr
        +v(\gamma(s_0),s_0) < v(x_0,t_0)+1
        \right\}.
    \end{align*}
Then, by Proposition \ref{thm:ex-DPP} and standard computations for value functions
without stopping time,
we have 
\begin{align}
    v(x_0,t_0)
    &=\inf_{\gamma\in\widetilde{\mathcal{A}}(x_0,t_0)}
    \left\{
    \int_{s_0}^{t_0} L(\gamma(r),\dot{\gamma}(r))dr + v(\gamma(s_0),s_0)
    \right\}\notag\\
    &=\inf_{\gamma\in\widetilde{\mathcal{A}}(x_0,t_0)}
    \left\{
    \int_{s}^{t_0} L(\gamma(r),\dot{\gamma}(r))dr + v(\gamma(s),s)
    \right\}\label{eq:Bellman-eq-5}
\end{align}
for all $s\in[s_0,t_0]$.
Since $(v-\phi)(x_0,t_0)=\min(v-\phi)$,
for all $\gamma\in\widetilde{\mathcal{A}}(x_0,t_0)$
and $s\in[s_0,t_0]$,
it follows that
\begin{align*}
    v(x_0,t_0)-v(\gamma(s),s)
    \leq \phi (x_0,t_0) - \phi(\gamma(s),s)
    =\int_{s}^{t_0} \left\{
    \phi_t(\gamma(r),r)+D\phi(\gamma(r),r)\cdot \dot{\gamma}(r)
    \right\}dr,
\end{align*}
and that
\begin{align*}
    &v(x_0,t_0)- 
    \left\{
    \int_{s}^{t_0}L(\gamma(r),\dot{\gamma}(r))dr
    +v(\gamma(s),s)\right\}\notag\\
    &\leq 
    \int_{s}^{t_0} \left\{
    \phi_t(\gamma(r),r)+H(\gamma(r),D\phi(\gamma(r),r)
    \right\}dr.
\end{align*}
Applying (\ref{eq:Bellman-eq-5}) and taking $\sup_{\gamma\in\widetilde{\mathcal{A}}(x_0,t_0)}$,
we have 
\begin{align}
    0\leq \sup_{\gamma\in\widetilde{\mathcal{A}}(x_0,t_0)}
    \left\{
    \int_{s}^{t_0} \left\{
    \phi_t(\gamma(r),r)+H(\gamma(r),D\phi(\gamma(r),r)
    \right\}dr
    \right\}.\label{eq:Bellman-eq-8}
\end{align}
Here, each $\gamma\in\widetilde{\mathcal{A}}(x_0,t_0)$
has a uniform continuity estimate, that is,
there exists $C>0$ such that
\begin{align}
    |x_0 -\gamma (s)|\leq C(t_0-s)^{1/2}\quad
    \text{for all } \gamma\in \widetilde{\mathcal{A}}(x_0,t_0)
    \text{ and } s\in[s_0,t_0].\label{eq:Bellman-eq-6}
\end{align}
Indeed,
(\ref{assump:L-quad-growth}) implies
\begin{align*}
    &|x_0 - \gamma(s)|
    \leq \int_{s}^{t_0} |\dot{\gamma}(r)|dr
    \leq (t_0-s)^{1/2}\left(
        \int_{s}^{t_0} |\dot{\gamma}(r)|^2dr
    \right)^{1/2}\\
    &\leq (t_0-s)^{1/2}\left(
        \int_{s_0}^{t_0}|\dot{\gamma}(r)|^2dr
    \right)^{1/2}
    \leq (t_0-s)^{1/2}\left(
        \int_{s_0}^{t_0} (2L(\gamma(r),\dot{\gamma}(r)) +K_0)dr
    \right)^{1/2}\\
    &\leq (t_0-s)^{1/2}
    \left(
        2 \left\{
            v(x_0,t_0)+1-\inf_{z\in\R^n,r\geq 0}v(z,r)
        \right\} +K_0(t_0-s_0)
    \right)^{1/2}
    \leq C (t_0-s)^{1/2}
\end{align*}
for all $\gamma \in \widetilde{\mathcal{A}}(x_0,t_0)$
and $s\in[s_0,t_0]$.
Moreover,
by the continuity of 
$(x,t)\mapsto\phi_t +H(x,D\phi)$,
for every $\lambda >0$,
there exists $s_1\in[s_0,t_0]$ such that 
\begin{align}
    |\phi_t(z,r)+H(z,D\phi(z,r))-\phi_t(x_0,t_0)-H(x_0,D\phi(x_0,t_0))|
    <\lambda \label{eq:Bellman-eq-7}
\end{align}
for all $r\in[s_1,t_0]$ and $|z-x_0|\leq C(t_0-r)^{1/2}$.
From (\ref{eq:Bellman-eq-6}) and (\ref{eq:Bellman-eq-7}),
for all $\lambda>0$, if $s$ is close enough to $t_0$,
it holds that
\begin{align*}
    \int_s^{t_0} \{\phi_t(\gamma(r),r)
    +H(\gamma(r),D\phi(\gamma(r),r))\}dr
    \leq \int_s^{t_0} \{\phi_t(x_0,t_0)+H(x_0,D\phi(x_0,t_0))+\lambda\}dr
\end{align*}
for all $\gamma\in\widetilde{\mathcal{A}}(x_0,t_0)$.
By (\ref{eq:Bellman-eq-8}),
dividing both sides by $t_0-s$ and taking $\sup_{\gamma\in\widetilde{\mathcal{A}}(x_0,t_0)}$,
we have 
\begin{align*}
    0&\leq \frac{1}{t_0-s} \int_s^{t_0}
    \{\phi_t(x_0,t_0)+H(x_0,D\phi(x_0,t_0))+\lambda\}dr\\
    &= \phi_t(x_0,t_0)+H(x_0,D\phi(x_0,t_0))+\lambda.
\end{align*}
Finally, letting $\lambda \to +0$, we obtain
$\phi_t(x_0,t_0) + H(x_0, D\phi(x_0,t_0)) \geq 0.$
In conclusion,
$v$ satisfies
\begin{align*}
    \max\{v_t+H(x,Dv),v-\psi(x)\} \geq 0
\end{align*}
in the sense of viscosity solutions.
\end{proof}

Applying Proposition \ref{thm:Bellman-eq} to (\ref{OP}),
we have optimal control formulas for our original problem (\ref{OP}) and effective equation (\ref{effOP}).

For notational simplicity,
we define functions $f_\varepsilon, f: \R^n \times [0,\infty)\to\R$
by
\begin{align*}
    f_\varepsilon (x,t)=\left\{
    \begin{aligned}
        &u_0(x) &&\text{if } t=0,\\
        &\psi\left(x,\frac{x}{\varepsilon}\right)
        &&\text{if } t>0,
    \end{aligned}
    \right.
    \quad
    f(x,t)=\left\{
    \begin{aligned}
        &u_0(x) &&\text{if } t=0,\\
        &\overline{\psi}(x) &&\text{if } t>0.
    \end{aligned}
    \right.
\end{align*}
\begin{proposition}
    Assume $(\mathrm{A}\ref{assump:H-peri})$--$(\mathrm{A}\ref{assump:initial})$.
    Then, the unique viscosity solution $u^\varepsilon$ and $u$ to $($\ref{OP}$)$ and $($\ref{effOP}$)$ 
    satisfies 
    \begin{align}
    u^\varepsilon(x,t)
    = \inf_{\substack{\gamma\in\mathcal{A}(x,t)\\ \theta\in[0,t]}}
    \left\{\int_{\theta}^t L\left(
    \frac{\gamma(s)}{\varepsilon}, \dot{\gamma}(s)
    \right)ds + f_\varepsilon(\gamma(\theta),\theta)\right\}
    \label{eq:u-ep-optim-ctrl-form}
    \end{align}
    and
    \begin{align}
    u(x,t)
    =\inf_{\substack{\gamma\in\mathcal{A}(x,t)\\\theta\in[0,t]}}
    \left\{\int_\theta^t \overline{L}(\dot{\gamma}(s))ds
    +f(\gamma(\theta),\theta)\right\},
    \label{eq:u-optim-ctrl-form}
    \end{align}
    respectively.
\end{proposition}

Moreover, 
since the Hamiltonian $\overline{H}$ is homogenous,
we have the following Hopf--Lax-type formula for the solution to (\ref{effOP}).
\begin{proposition}[{\cite[Proposition 2.1]{HT25}}]
    Assume $(\mathrm{A}\ref{assump:H-peri})$--$(\mathrm{A}\ref{assump:initial})$.
    Then, the unique viscosity solution $u$ to $($\ref{effOP}$)$ satisfies
    \begin{align}
        u(x,t)
        =\inf_{\substack{
            z\in\R^n\\
            \theta\in[0,t)
        }}\left\{
          (t-\theta)\overline{L}\left(
              \frac{x-z}{t-\theta}
          \right) + f(z,\theta)
        \right\}
        \label{thm:Hopf-Lax}
    \end{align}
    for all $(x,t)\in\R^n\times (0,\infty)$.
\end{proposition}
A proof of this formula was given in \cite{HT25} 
by directly verifying the dynamic programming principle
for the right-hand side of (\ref{thm:Hopf-Lax}).
Here, 
we provide an alternative proof of this representation formula based on (\ref{eq:u-optim-ctrl-form}).

\begin{proof}
    Fix $(x,t)\in \R^n\times (0,\infty)$, 
    and take $z\in\R^n$ and $\theta\in[0,t)$ arbitrarily.
    By considering the path $\gamma_{z,\theta}\in\mathcal{A}(x,t)$ defined by
    \begin{align*}
        \gamma_{z,\theta}(s)=z+\frac{s}{t-\theta}(x-z),
        \quad s\in[0,t]
    \end{align*}
    for each $z\in\R^n$ and $\theta\in[0,t)$,
    we have
    \begin{align*}
        u(x,t)
        \leq \int_\theta^t \overline{L}\left(
            \frac{x-z}{t-\theta}
        \right)ds + f(z,\theta)
        = (t-\theta)\overline{L}\left(
            \frac{x-z}{t-\theta}
        \right) + f(z,\theta)
    \end{align*}
    for all $z\in\R^n$ and $\theta\in[0,t)$.
    Then, the inequality
    \begin{align*}
        u(x,t)
        \leq\inf_{\substack{
            z\in\R^n\\
            \theta\in[0,t)
        }}\left\{
          (t-\theta)\overline{L}\left(
              \frac{x-z}{t-\theta}
          \right) + f(z,\theta)
        \right\}
    \end{align*}
    holds.
    
    On the other hand,
    for any $\delta>0$,
    let $(\gamma_\delta,\theta_\delta)\in\mathcal{A}(x,t)\times[0,t]$ be the control such that
    \begin{align*}
        \int_{\theta_\delta}^t \overline{L}\left(
            \dot{\gamma}_\delta(s)
        \right)ds + f(\gamma_\delta(\theta_\delta),\theta_\delta )
        < u(x,t) +\delta.
    \end{align*}
    If $\theta_\delta <t$,
    we have
    \begin{align*}
        \int_{\theta_\delta}^t \overline{L}\left(
            \dot{\gamma}_\delta(s)
        \right)ds
        \geq (t-\theta_\delta) \overline{L}\left(
            \frac{x-\gamma_\delta(\theta_\delta)}{t-\theta_\delta}
        \right)
    \end{align*}
    from Jensen's inequality since $\overline{L}$ is convex.
    Then, it follows that
    \begin{align*}
        u(x,t)+\delta
        &> (t-\theta_\delta) \overline{L}\left(
            \frac{x-\gamma_\delta(\theta_\delta)}{t-\theta_\delta}
        \right) + f(\gamma_\delta(\theta_\delta),\theta_\delta)\\
        &\geq \inf_{\substack{
            z\in\R^n\\
            \theta\in[0,t)
        }}\left\{
          (t-\theta)\overline{L}\left(
              \frac{x-z}{t-\theta}
          \right) + f(z,\theta)
        \right\}.
    \end{align*}
    Otherwise if $\theta_\delta=t$, 
    then it follows that
    \begin{align*}
        u(x,t)+\delta 
        &> \overline{\psi}(x)
        =\lim_{\theta\uparrow t}\left\{
            (t-\theta)\overline{L}(0)+f(x,\theta)
        \right\}\\
        &\geq \liminf_{\theta\uparrow t}
        \inf_{\substack{
            z\in\R^n
        }}\left\{
          (t-\theta)\overline{L}\left(
              \frac{x-z}{t-\theta}
          \right) + f(z,\theta)
        \right\}\\
        &\geq \inf_{\substack{
            z\in\R^n\\
            \theta\in[0,t)
        }}\left\{
          (t-\theta)\overline{L}\left(
              \frac{x-z}{t-\theta}
          \right) + f(z,\theta)
        \right\}.
    \end{align*}
    By letting $\delta\to+0$, we have
    \begin{align*}
        u(x,t)\geq \inf_{\substack{
            z\in\R^n\\
            \theta\in[0,t)
        }}\left\{
          (t-\theta)\overline{L}\left(
              \frac{x-z}{t-\theta}
          \right) + f(z,\theta)
        \right\}.
    \end{align*}
\end{proof}

We conclude this section 
by giving a uniform Lipschitz estimate of minimizing trajectory of optimal stopping problems,
which we need in our proof of Theorem \ref{thm:main-thm}.
\begin{proposition}\label{prop:lip-minimizer}
    Assume $(\mathrm{A}\ref{assump:H-peri})$--$(\mathrm{A}\ref{assump:initial})$.
    Then, for each $\varepsilon>0$ and $(x,t)$,
    there exists a minimizer $(\gamma, \theta)\in \mathcal{A}(x,t)\times[0,t]$ 
    which attains the infimum of the right-hand side of $(\ref{eq:u-ep-optim-ctrl-form})$.
    Moreover, there exists a constant $C>0$ depending only on
    $L$ and $\lVert Du_0 \rVert_{L^\infty(\R^n)}$ such that
    $|\dot{\gamma}(s)|\leq C$ for a.e. $s\in[\theta ,t ]$.
\end{proposition}

\begin{proof}
    For $\varepsilon>0$,
    $(x,t)\in\R^n\times [0,\infty)$ and $\theta\in[0,t]$,
    we define
    \begin{align}
        V^\varepsilon (x,t,\theta)
        := \inf_{\gamma\in\mathcal{A}(x,t)}
    \left\{\int_{\theta}^t L\left(
    \frac{\gamma(s)}{\varepsilon}, \dot{\gamma}(s)
    \right)ds + f_\varepsilon(\gamma(\theta),\theta)\right\}.
    \label{eq:lip-minimizer-pf1}
    \end{align}
    We fix $\varepsilon$ and $(x,t)$.
    Then, 
    it follows from the general existence theorem of minimizers for action functionals that,
    for each $\theta\in[0,t]$,
    there exists a minimizer $\gamma_\theta\in\mathcal{A}(x,t)$ 
    that attains the infimum on the right-hand side of (\ref{eq:lip-minimizer-pf1});
    see, for example,
    \cite[Appendix D]{MR4328923}.
    Since $V^\varepsilon (x,t,\theta)$ is lower semicontinuous in $\theta\in[0,t]$ from Proposition \ref{thm:lsc-in-theta},
    we obtain the existence of minimizing pair $(\gamma_{0},\theta_0)\in\mathcal{A}(x,t)\times[0,t]$, which satisfies
    \begin{align*}
        u^\varepsilon (x,t)
        = \int_{\theta_0}^t L\left(
    \frac{\gamma_{0}(s)}{\varepsilon}, \dot{\gamma}_{\theta_0}(s)
    \right)ds + f_\varepsilon(\gamma_{0}(\theta_0),\theta_0).
    \end{align*}
    For each $s\in[\theta_0,t]$,
    it follows that
    \begin{align*}
        u^\varepsilon (x,t)
        &=\int_{\theta_0}^s L\left(
            \frac{\gamma_0(r)}{\varepsilon},
            \dot{\gamma}_0(r)
        \right)dr + \int_s^t L\left(
            \frac{\gamma_0(r)}{\varepsilon},
            \dot{\gamma}_0(r)
        \right)dr 
        + f_\varepsilon (\gamma_0(\theta_0),\theta_0)\\
        &\geq \int_s^t L\left(
            \frac{\gamma_0(r)}{\varepsilon},
            \dot{\gamma}_0(r)
        \right)dr
        +u^\varepsilon (\gamma_0(s),s).
    \end{align*}
    Then, by the dynamic programming principle,
    we have 
    \begin{align}
        u^\varepsilon (x,t)
        &= \int_s^t L\left(
            \frac{\gamma_0(r)}{\varepsilon},
            \dot{\gamma}_0(r)
        \right)dr
        +u^\varepsilon (\gamma_0(s),s)
        \label{eq:lip-minimizer-pf2}
    \end{align}
    for each $s\in[\theta_0,t]$.
    Evaluating (\ref{eq:lip-minimizer-pf2}) at $s$ and $s+h$
    with $s,s+h \in[\theta_0,t]$,
    and taking difference, 
    we obtain
    \begin{align*}
        u^\varepsilon (\gamma_0(s+h),s+h)
        - u^\varepsilon (\gamma_0(s),s)
        =\int_s^{s+h} L\left(
            \frac{\gamma_0(r)}{\varepsilon},
            \dot{\gamma}_0(r)
        \right)dr.
    \end{align*}
    Since $L$ satisfies (\ref{assump:L-quad-growth}),
    it follows from Jensen's inequality that
    \begin{align}
        \frac{
            u^\varepsilon (\gamma_0(s+h),s+h)- u^\varepsilon (\gamma_0(s),s)
        }{h}
        &=\frac{1}{h}\int_s^{s+h} L\left(
            \frac{\gamma_0(r)}{\varepsilon},
            \dot{\gamma}_0(r)
        \right)dr\notag \\
        &\geq \frac{1}{h}
        \int_s^{s+h} \left(
            \frac{1}{2}|\dot{\gamma}_0(r)|^2 - K_0
        \right)dr\notag \\
        &\geq \frac{1}{2}\left|
            \frac{\gamma_0(s+h)-\gamma_0(s)}{h}
        \right|^2 -K_0.
        \label{eq:lip-minimizer-pf3}
    \end{align}
    On the other hand,
    by Proposition \ref{thm:a-priori},
    there exists a constant $C_0>0$,
    depending only on $H$ and $\lVert Du_0 \rVert_{L^\infty(\R^n)}$,
    such that
    \begin{align}
        \frac{
            u^\varepsilon (\gamma_0(s+h),s+h) - u^\varepsilon (\gamma_0(s),s)
        }{h}
        \leq C_0 \left(
            \left|
                \frac{\gamma_0(s+h)-\gamma_0(s)}{h}
            \right| +1
        \right).
        \label{eq:lip-minimizer-pf4}
    \end{align}
    Combining (\ref{eq:lip-minimizer-pf3}) 
    and (\ref{eq:lip-minimizer-pf4}),
    and taking $\lim_{h\to 0}$ if $\gamma_0$ is differentiable at $s$,
    we have
    \begin{align*}
        \frac{1}{2}|\dot{\gamma}_0(s)|^2 - K_0
        \leq C_0 (|\dot{\gamma}_0(s)|+1).
    \end{align*}
    Thus, we obtain the estimate
    \begin{align*}
        |\dot{\gamma}_0(s)| \leq C_0 +\sqrt{C_0^2 +2 (K_0+C_0)}
        \quad \text{for a.e. } s\in [\theta_0,t],
    \end{align*}
    with a constant independent of $\varepsilon$ and $(x,t)$.
    This completes the proof.
\end{proof}

\section{Homogenization results}
Recall the metric function, which we introduced in Section 1.
For every $x,y\in\R^n$ and $t>0$, we define
\begin{align}\label{def:metric-fcn}
    m(t,x,y)
    :=\inf\left\{
    \int_0^t L(\eta(s),\dot{\eta}(s))ds
    \; : \; \eta\in\AC ([0,t]; \R^n),\ \eta(0)=x,\ \eta(t)=y
    \right\}.
\end{align}
The value
$m(t,x,y)$ represents the minimum cost traveling from $x$ to $y$ 
in a given time $t$.
Using (\ref{def:metric-fcn}) and the optimal control formula, 
we can caluculate the unique viscosity solution $u^\varepsilon$ to (\ref{OP}) as
\begin{align*}
    u^\varepsilon(x,t)
    &=\inf_{\substack{
        \eta\in\mathcal{A}(x/\varepsilon,t/\varepsilon)\\
        \theta\in[0,t]
    }}
    \left\{
        \varepsilon\int_{\theta/\varepsilon}^{t/\varepsilon}
            L(\eta(s),\dot{\eta}(s))
        ds + f_\varepsilon\left(
            \varepsilon\eta\left(
                \frac{\theta}{\varepsilon}
            \right), \theta
        \right)
    \right\}\\
    &=\inf_{\substack{
        z\in\R^n\\ 
        \theta\in (0,t)
    }}\left\{
        \varepsilon m\left(
            \frac{t-\theta}{\varepsilon},
            \frac{z}{\varepsilon},
            \frac{x}{\varepsilon}
        \right) +f_\varepsilon (z,\theta)
    \right\}\\
    &=\min\left\{
        \inf_{\substack{
            z\in\R^n\\
            \theta\in (0,t)
        }}\left\{
            \varepsilon m\left(
                \frac{t-\theta}{\varepsilon},
                \frac{z}{\varepsilon},
                \frac{x}{\varepsilon}
            \right) +\psi \left(
                z, \frac{z}{\varepsilon}
            \right)
        \right\},\ 
        \inf_{z\in\R^n}\left\{
            \varepsilon m\left(
                \frac{t}{\varepsilon},
                \frac{z}{\varepsilon},
                \frac{x}{\varepsilon}
            \right) +u_0(z)
        \right\}
    \right\}.
\end{align*}

\subsection{Proof of homogenization}

First, we review the following two properties for the metric function $m$.
Is is known that the metric function $m$ has the local Lipschitz property.

\begin{lemma}[\cite{MR2361998}, Theorem 3.1]\label{lem:lip-m}
    For every $M>0$,
    there exists a constant $C>0$ depending only on $n$, $L$
    and $M$ such that $m$ is Lipschitz continuous 
    on the domain
    \begin{align*}
        \{(t,x,y)\in[0,\infty)\times\R^n\times\R^n
        \; : \; |x-y|\leq Mt\}
    \end{align*}
    with $C$ being a Lipschitz bound.
\end{lemma}

According to \cite{MR4946874}, 
$m$ has the following quantitative property for its homogenization limit.

\begin{lemma}[\cite{MR4946874}]\label{lem:homo-metric}
    For each $t>0$, $x,y\in\R^n$, 
    there exists the limit
    \begin{align}
        \overline{m}(t,x,y):=\lim_{k\to\infty} 
        \frac{1}{k}m(kt,kx,ky).
    \end{align}
    Moreover, for every $M>0$,
    there exists a constant $C>0$ depending only on $n$,
    $L$ and $M$ such that 
    \begin{align}
        |m(t,x,y)-\overline{m}(t,x,y)|\leq C
        \quad\text{on } 
        \{(t,x,y)\in[0,\infty)\times\R^n\times\R^n
        \; : \; |x-y|\leq Mt\}.
    \end{align}
\end{lemma}

Thanks to the theory of periodic homogenization and Hopf--Lax formula for usual Cauchy problems,
we have 
\begin{align}
    \overline{m}(t,x,y)=t\overline{L}\left(
    \frac{y-x}{t}\right).
\end{align}
Then, the unique viscosity solution $u$ to (\ref{effOP})
has the formula
\begin{align*}
    u(x,t)&=\inf_{\substack{z\in\R^n\\ \theta\in[0,t)}}
    \left\{
    (t-\theta)\overline{L}\left(\frac{x-z}{t-\theta}\right)
    +f(z,\theta)
    \right\}\\
    &=\min\left\{
    \inf_{\substack{z\in\R^n\\ \theta\in (0,t)}}
    \left\{\overline{m}\left(t-\theta,
    z,
    x\right)
    +\overline{\psi} \left(z\right)\right\},\ 
    \inf_{z\in\R^n}\left\{
    \overline{m}\left(t,
    z,x\right)
    +u_0(z)\right\}
    \right\}.
\end{align*}

We now prove Theorem \ref{thm:main-thm}.

\begin{proof}[Proof of Theorem \ref{thm:main-thm}]

We start from proving that
there exists a constant $C>0$
such that the inequality
\begin{align}
    \inf_{\substack{x\in\R^n\\t\in[0,\infty)}}\left\{
    u^\varepsilon(x,t)-u(x,t)
    \right\}
    \geq -C\varepsilon\label{eq:3.2-6}
\end{align}
holds for every $\varepsilon>0$.
Take $\varepsilon >0$ and $(x,t)\in \R^n\times [0,\infty)$,
then there exists $z_\varepsilon,\theta_\varepsilon$ satisfying
\begin{align*}
    u^\varepsilon (x,t) 
    = \varepsilon m\left(
    \frac{t-\theta_\varepsilon}{\varepsilon},
    \frac{z_\varepsilon}{\varepsilon},
    \frac{x}{\varepsilon}\right)
    + f_\varepsilon (z_\varepsilon,\theta_\varepsilon).
\end{align*}
By Proposition \ref{prop:lip-minimizer},
we have $(z_\varepsilon ,\theta_\varepsilon)=(x,t)$ or that
\begin{align*}
    \frac{|x-z_\varepsilon|}{t-\theta_\varepsilon}
    = \frac{|x/\varepsilon - z_\varepsilon /\varepsilon|}{(t-\theta_\varepsilon)/\varepsilon}
    \leq C_0
\end{align*}
for a constant $C_0>0$ depending only on $L$ 
and $\lVert Du_0\rVert_{L^\infty (\R^n)}$.
Hence, applying Lemma \ref{lem:homo-metric}
and noting that $f_\varepsilon\geq f$ by its definition,
we obtain
\begin{align*}
    u^\varepsilon (x,t)
    =\varepsilon m\left(
    \frac{t-\theta_\varepsilon}{\varepsilon},
    \frac{z_\varepsilon}{\varepsilon},
    \frac{x}{\varepsilon}\right)
    + f_\varepsilon (z_\varepsilon,\theta_\varepsilon)
    &\geq \overline{m}(t-\theta_\varepsilon,z_\varepsilon,x)
    -C\varepsilon + f(z_\varepsilon,\theta_\varepsilon)\\
    &\geq u(x,t) -C\varepsilon.
\end{align*}
Thus, (\ref{eq:3.2-6}) follows.

Next, we prove that
\begin{align}
    \limsup_{\varepsilon\to +0} 
    \sup\left\{
    u^\varepsilon(x,t)-u(x,t)
    \; : \; (x,t)\in B(0,R) \times [0,T]\right\}
    \leq 0 \label{eq:3.2-5}
\end{align}
for each $R,T>0$.
Take $(x,t)\in B(0,R) \times [0,T]$ arbitrarily,
then there exists $z,\theta$ satisfying
\begin{equation*}
    u (x,t) = \overline{m}(t-\theta, z, x) + f(z,\theta).
\end{equation*}
Similarly, by Proposition \ref{prop:lip-minimizer},
it follows that $|x-z|\leq C_0 (t-\theta)$ for a constant $C_0>0$ depending only on $L$ 
and $\lVert Du_0\rVert_{L^\infty (\R^n)}$.
\begin{customdefinition}{Case 1}
    The case that $\theta =0$.
\end{customdefinition}
Since $|x-z|\leq C_0 (t-\theta)=C_0 t$,
applying Lemma \ref{lem:homo-metric},
we have
    \begin{align}
        u^\varepsilon (x,t)
        &\leq \varepsilon m\left(
        \frac{t}{\varepsilon},\frac{z}{\varepsilon},
        \frac{x}{\varepsilon} \right) + u_0 (z)
        \leq \overline{m}(t,z,x) + C\varepsilon +u_0(z)
        = u(x,t) + C\varepsilon\label{eq:3.2-2}
    \end{align}
    for each $\varepsilon >0$.
\begin{customdefinition}{Case 2}
    The case that $0<\theta \leq t$.
\end{customdefinition}
For each $\varepsilon >0$, it follows that
\begin{align}\label{eq:3.2-1}
        u(x,t)
        &= \overline{m}(t-\theta, z, x) + \psi\left(
        z, \frac{\widetilde{z}}{\varepsilon}
        \right)
        \geq \varepsilon m\left(
        \frac{t-\theta}{\varepsilon}, \frac{z}{\varepsilon},
        \frac{x}{\varepsilon}
        \right) 
        -C\varepsilon +\psi \left(
        z, \frac{\widetilde{z}}{\varepsilon}
        \right),
    \end{align}
    where we let $\widetilde{z}\in\R^n$ be a point satisfying
    \begin{align*}
        \frac{z-\widetilde{z}}{\varepsilon}\in \left[
        -\frac{1}{2},\frac{1}{2}
        \right]^n,
        \quad \overline{\psi}(z) = \psi\left(
        z, \frac{\widetilde{z}}{\varepsilon}\right).
    \end{align*}
Then, we consider following two cases:
\begin{itemize}
    \item For $\varepsilon >0$
    with $\varepsilon \geq t-\theta$,
    since $|x-z|\leq C_0|t-\theta|\leq C_0\varepsilon$,
    we have
    \begin{align*}
        u^\varepsilon(x,t)
        &\leq u^\varepsilon (\widetilde{z},t)
            +\lVert 
                Du^\varepsilon
            \rVert_{L^\infty(\R^n\times [0,\infty))}
            (|x-z|+|z-\widetilde{z}|)\\
        &\leq \psi\left(
            \widetilde{z},\frac{\widetilde{z}}{\varepsilon}
        \right)
        +\lVert 
            Du^\varepsilon
        \rVert_{L^\infty(\R^n\times [0,\infty))}
        (C_0 +\sqrt{n})\varepsilon
        \leq \psi\left(
            \widetilde{z},\frac{\widetilde{z}}{\varepsilon}
        \right) + C\varepsilon,
    \end{align*}
    and by Jensen's inequality, we have
    \begin{align*}
        \varepsilon m\left(
            \frac{t-\theta}{\varepsilon}, 
            \frac{z}{\varepsilon},
            \frac{x}{\varepsilon}
        \right)
        &\geq
        \inf_{
            \substack{
                \eta\in\mathcal{A}(x/\varepsilon,(t-\theta)/\varepsilon)\\
                \eta(0)=z/\varepsilon
            }
        }
        \left\{
            \varepsilon\int_0^{\frac{t-\theta}{\varepsilon}}\left(
                \frac{1}{2}|\dot{\eta}(s)|^2 - K_0
            \right)ds
        \right\}\\
        &\geq (t-\theta)\left(
            \frac{1}{2}\left|
                \frac{x-z}{t-\theta}
            \right|^2 -K_0
        \right)
        \geq -(t-\theta)K_0 
        \geq -K_0 \varepsilon.
    \end{align*}
    Substituting them into (\ref{eq:3.2-1}),
    we have
    \begin{align}
        u(x,t)
        &\geq -K_0\varepsilon -C\varepsilon +\psi\left(\widetilde{z},\frac{\widetilde{z}}{\varepsilon}\right)
        -\left|
        \psi\left(z,\frac{\widetilde{z}}{\varepsilon}\right)
        -\psi\left(\widetilde{z},\frac{\widetilde{z}}{\varepsilon}\right)
        \right|\notag \\
        &\geq u^\varepsilon (x,t)
        -C\varepsilon - \omega_{R+C_0 T +1} (\sqrt{n}\varepsilon).\label{eq:3.2-3}
    \end{align}
    Here, $\omega_r(\cdot)$ for $r>0$ denotes a modulus of continuity of $\psi$ on $B(0,r)\times \R^n$.
    \item For $\varepsilon>0$ 
    with $\varepsilon < t-\theta$,
    it follows that
    \begin{align*}
        \left|
        \frac{x}{\varepsilon}-\frac{\widetilde{z}}{\varepsilon}
        \right|
        \leq \frac{|x-z|}{\varepsilon}
        +\frac{|z-\widetilde{z}|}{\varepsilon}
        \leq (C_0 + \sqrt{n})\frac{t-\theta}{\varepsilon}.
    \end{align*}
    Then, 
    by Lemma \ref{lem:lip-m},
    we have 
    \begin{align*}
        \left|m\left(
        \frac{t-\theta}{\varepsilon}, \frac{z}{\varepsilon},
        \frac{x}{\varepsilon}
        \right)
        -m\left(
        \frac{t-\theta}{\varepsilon}, 
        \frac{\widetilde{z}}{\varepsilon},
        \frac{x}{\varepsilon}
        \right)\right|
        \leq C\frac{|z-\widetilde{z}|}{\varepsilon}
        \leq C.
    \end{align*}
    Thus, by (\ref{eq:3.2-1}), we have
    \begin{align}
        u(x,t)
        &\geq \varepsilon m\left(
        \frac{t-\theta}{\varepsilon}, 
        \frac{\widetilde{z}}{\varepsilon},
        \frac{x}{\varepsilon}
        \right)
        -C\varepsilon 
        +\psi\left(\widetilde{z},\frac{\widetilde{z}}{\varepsilon}\right)
        -\left|
        \psi\left(z,\frac{\widetilde{z}}{\varepsilon}\right)
        -\psi\left(\widetilde{z},\frac{\widetilde{z}}{\varepsilon}\right)
        \right|\notag \\
        &\geq u^\varepsilon (x,t)
        -C\varepsilon 
        -\omega_{R+C_0 T +1} (\sqrt{n}\varepsilon).\label{eq:3.2-4}
    \end{align}
\end{itemize}

From (\ref{eq:3.2-2}), (\ref{eq:3.2-3}) and (\ref{eq:3.2-4}),
we conclude (\ref{eq:3.2-5}).
Moreover, if $\psi\in\Lip (\R^n\times\R^n)$,
we immediately obtain that
$\lVert u^\varepsilon - u\rVert_{L^\infty (\R^n\times[0,\infty))}\leq C \varepsilon$.
\end{proof}

\subsection{Examples}
\begin{example}[Optimality of the convergence rate $O(\varepsilon)$]
Using the example for a standard Cauchy problem that attains the optimal convergence rate $O(\varepsilon)$,
introduced in \cite{MR3951696},
we can readily construct an example in our setting that attains the convergence rate $O(\varepsilon)$ by adding a suitable obstacle function.
This shows that $O(\varepsilon)$ is also the optimal convergence rate in our setting.

Let $n=1$ and $V\in C(\R)$ be the $\Z$-periodic function
satisfying
\begin{align*}
    \min_{\R} V=0
    \quad \text{and} \quad
    V\geq 1 \text{ on } 
    \left[-\frac{1}{3},\frac{1}{3}\right].
\end{align*}
Then, we define $H(y,p)=\frac{1}{2}|p|^2 - V(y)$
and consider the obstacle problem:
\begin{align}
    \left\{
        \begin{aligned}
            &\max\left\{
                u^\varepsilon_t+\frac{1}{2}|u^\varepsilon_x|^2-V\left(\frac{x}{\varepsilon}\right),\ 
                u^\varepsilon - M
            \right\}
            =0 &&\text{in } \R\times(0,\infty),\\
            &u^\varepsilon (x,0)=0
            &&\text{on } \R,
        \end{aligned}
    \right.
    \label{eq:optim-ex}
\end{align}
with a large constant obstacle function $M>0$
such that $M>\max_\R V +2$.
In this setting,
we have $L(y,q)=\frac{1}{2}|q|^2 + V(y)$ and
that the unique viscosity solution $u^\varepsilon$ to (\ref{eq:optim-ex}) can be written as
\begin{align}
    u^\varepsilon(x,t)
    =\inf_{\substack{
        \eta\in\mathcal{A}(x/\varepsilon,t/\varepsilon)\\
        \theta\in[0,t]
    }}\left\{
        \varepsilon \int_{\theta/\varepsilon}^{t/\varepsilon}
        \left(
            \frac{1}{2}|\dot{\eta}(s)|^2+V(\eta(s))
        \right)ds
        +\left\{
            \begin{aligned}
                &0 &&\text{if } \theta=0,\\
                &M &&\text{if } \theta\in (0,t]
            \end{aligned}
        \right.
    \right\}.
    \label{eq:ctrl-form-optim-ex}
\end{align}
Here, 
for any small $\varepsilon>0$,
we claim that the infimum in the right-hand side of (\ref{eq:ctrl-form-optim-ex}) can be attained 
only at pairs $(\eta,\theta)$ with $\theta=0$.
Indeed, 
for any $(x,t)$ with $t\leq 1$,
we have 
\begin{align}
    u^\varepsilon (x,t)
    \leq \varepsilon \int_0^{t/\varepsilon} 
    V\left(\frac{x}{\varepsilon}\right)ds
    =t V\left(\frac{x}{\varepsilon}\right)
    \leq \max_\R V,
\end{align}
by considering a control $(\eta,\theta)$ such that
$\eta(s)\equiv x$ and $\theta=0$.
Otherwise, for any $(x,t)$ with $t>1$,
we consider the control $(\eta,\theta)$ such that
$\theta=0$ and
\begin{align*}
    \eta(s)=\left\{
        \begin{aligned}
            &\frac{\widetilde{x}}{\varepsilon}
            &&\text{if } 0\leq s \leq \frac{t}{\varepsilon}-1,\\
            &\frac{\widetilde{x}}{\varepsilon}
            +\left(
                s-\frac{t}{\varepsilon}+1
            \right)
            \frac{x-\widetilde{x}}{\varepsilon}
            &&\text{if }\frac{t}{\varepsilon}-1\leq s \leq\frac{t}{\varepsilon},
        \end{aligned}
    \right.
\end{align*}
where $\widetilde{x}\in\R$ is a point satisfying
\begin{align*}
    \frac{x-\widetilde{x}}{\varepsilon}\in 
    \left[-\frac{1}{2},\frac{1}{2}\right]
    \quad\text{and}\quad
    V\left(\frac{\widetilde{x}}{\varepsilon}\right)
    =\min_\R V \ (=0).
\end{align*}
Since $V(\eta(s))= 0$ for $s\in[0,\frac{t}{\varepsilon}-1]$, we have 
\begin{align*}
    u^\varepsilon (x,t)
    \leq \varepsilon \int_{\frac{t}{\varepsilon}-1}^{\frac{t}{\varepsilon}} \left(
        \frac{1}{2}\left|
            \frac{x-\widetilde{x}}{\varepsilon}
        \right|^2 +\max_\R V
    \right)ds 
    \leq\varepsilon \left(
        \frac{1}{2}+\max_\R V
    \right),
\end{align*}
and thus, we obtain $u^\varepsilon (x,t)< \max_\R V +1$
for every $(x,t)$.
On the other hand,
for any control $(\eta,\theta)$ with $\theta>0$,
the value of the cost funtional in the right-hand side of (\ref{eq:ctrl-form-optim-ex}) is always $\geq M$ 
since $V \geq 0$ in $\R$.
Moreover, since $M>\max_\R V + 2$ by its definition,
it holds that $u^\varepsilon$ satisfies
\begin{align}
    u^\varepsilon (x,t)
    =\inf_{\eta\in\mathcal{A}(x/\varepsilon,t/\varepsilon)}\left\{
        \varepsilon \int_{0}^{t/\varepsilon}
        \left(
            \frac{1}{2}|\dot{\eta}(s)|^2+V(\eta(s))
        \right)ds
    \right\}
    \label{eq:ctrl-form-optim-ex-2}
\end{align}
for every $(x,t)$.
It is shown in \cite[Proposition 4.3]{MR3951696}
that the function $u^\varepsilon$ of the form (\ref{eq:ctrl-form-optim-ex-2}) converges to $0$ on $\R\times[0,\infty)$ locally uniformly
as $\varepsilon\to +0$ and satisfies
$u^\varepsilon (0,1)\geq \frac{\varepsilon}{6}$
for any small $\varepsilon>0$.
Thus, the convergence rate of $u^\varepsilon$ to (\ref{eq:optim-ex}) is exactly $O(\varepsilon)$,
and no faster rate can be achieved.
\end{example}
\begin{example}
Next, we present an example that explains 
the convergence rate 
\begin{align*}
    \lVert u^\varepsilon -u\rVert_{L^\infty(\R^n\times[0,\infty))}= O(\varepsilon)
\end{align*}
in Theorem \ref{thm:main-thm} may fail to be held
if we do not assume Lipschitz continuity of an obstacle $\psi$.

For $\alpha\in(0,1)$, consider the obstacle problem:
\begin{align}\label{eq:OP-ex}
    \left\{
    \begin{aligned}
        &\max\left\{
        u^\varepsilon_t + |u^\varepsilon_x| -1,\ 
        u^\varepsilon - |\sin\pi x|^\alpha 
        - \left| \cos\frac{\pi x}{\varepsilon} \right|^\alpha
        \right\}=0
        &&\text{in }\R\times(0,\infty),\\
        &u^\varepsilon(x,0) = 0
        &&\text{on } \R.
    \end{aligned}
    \right.
\end{align}
In this case, we let $H(y,p)=H(p)=|p|-1$ and 
$\psi(x,y)=|\sin\pi x|^\alpha+|\cos \pi y |^\alpha$,
which is not Lipschitz but $\alpha$-H\"{o}lder continuous.
Here, 
we show a subsequence $\{u^{\varepsilon_k}\}$, $\varepsilon_k=1/k$, $k\in\N$, of unique viscosity solutions to (\ref{eq:OP-ex}) satisfies,
for its local uniform limit $u$
and a constant $c>0$,
\begin{align}
    u^{\varepsilon_k} \left(
        \frac{1}{2}, 2
    \right)
    -u\left(
        \frac{1}{2}, 2
    \right)
    \geq \frac{c}{k^\alpha}
    \label{eq:conv-order}
\end{align}
for all large $k\in\N$.
First, note that
$\overline{H}(p)=H(p)=|p|-1$, 
$\overline{\psi}(x)=|\sin\pi x|^\alpha$, and
\begin{align*}
    L(q)=\overline{L}(q)
    =\left\{
    \begin{aligned}
        &1 &&\text{if } -1\leq q \leq 1,\\
        &+\infty &&\text{if } |q|>1,
    \end{aligned}
    \right.
\end{align*}
and then, using the optimal control formula,
we can compute
\begin{align}
    u^{\varepsilon_k}(x,t)
    &= \inf\left\{
    t-\theta + \left\{
    \begin{aligned}
        &0 &&\text{if } \theta =0,\\
        &|\sin\pi x|^\alpha 
        + \left| \cos k\pi x \right|^\alpha
        &&\text{if } \theta\in(0,t]
    \end{aligned}
    \right.
    \; : \;
    \begin{aligned}
        &\gamma\in\AC([0,t]), \theta\in[0,t],\\
        &\gamma(t)=x, |\dot{\gamma}|\leq 1\text{ on } [0,t]
    \end{aligned}
    \right\}\notag\\
    &=\min\big\{
    t, \ 
    \inf 
    \left\{s+|\sin\pi z|^\alpha 
        + \left| \cos k\pi z \right|^\alpha
        \; : \; s\in[0,t), |z-x|\leq s
    \right\}
    \big\}.\label{eq:3.3-u-ep-formula}
\end{align}
Similarly, we can also compute
\begin{align}
    u(x,t)
    = \min\big\{
    t,\ 
    \inf\left\{
    s+ |\sin \pi z |^\alpha \; : \; 
    s\in [0,t), |z-x| \leq s
    \right\}
    \big\}.\label{eq:3.3-u-formula}
\end{align}

\begin{figure}[h]
\centering
\begin{tabular}{cc}
\includegraphics[height=5cm]{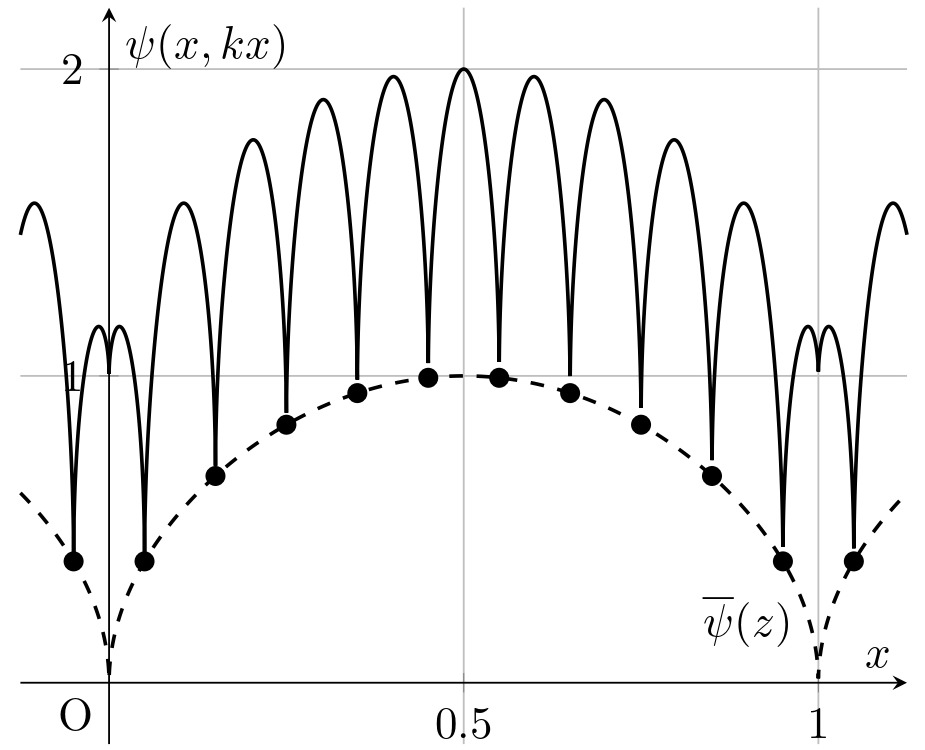} &
\includegraphics[height=5cm]{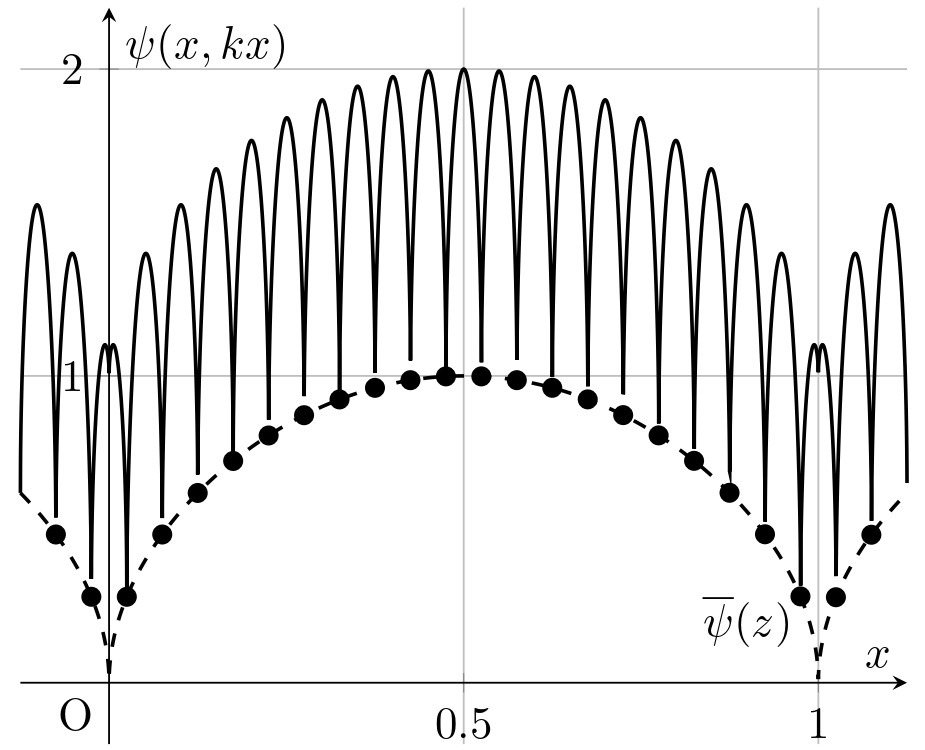}\\
$\alpha=1/2$, $k=10$ & $\alpha=1/2$, $k=20$
\end{tabular}
\caption{The graph of $\psi(z,kz)$}
\end{figure}

To prove (\ref{eq:conv-order}),
we focus on the value of $u^{\varepsilon_k}$ and $u$
at $(x,t)=(1/2,2)$.
When $t\geq2$,
one can see that the minimum of both 
(\ref{eq:3.3-u-ep-formula})
and (\ref{eq:3.3-u-formula})
are attained by the second terms
since
\begin{align*}
    &\inf\left\{
    s+|\sin\pi z|^\alpha +|\cos k \pi z|^\alpha
    \; : \; s\in [0,t),\; |z-x| \leq s
    \right\}
    \leq 0+ \left|\sin \frac{\pi}{2}\right|^\alpha 
    +\left| \cos \frac{k\pi}{2}\right|^\alpha
    \leq 2,\\
    &\inf\left\{
    s+|\sin\pi z|^\alpha
    \; : \; s\in [0,t),\; |z-x| \leq s
    \right\}
    \leq 0 + \left|\sin \frac{\pi}{2}\right|^\alpha
    \leq 1 <2.
\end{align*}
Using this estimate and symmetry of $\psi(z,kz)$ and $\overline{\psi}(z)$ with respect to $z=1/2$,
we can calculate as
\begin{align}
    u^{\varepsilon_k}\left(
    \frac{1}{2},2
    \right)
    &= \inf\left\{
    s + |\sin\pi z|^\alpha +|\cos k \pi z|^\alpha
    \; : \; s\in[0,2),\; z\in 
    \left[\frac{1}{2}-s,\frac{1}{2}+s\right]\cap
    \left[0,\frac{1}{2}\right]
    \right\}\notag\\
    &= \min\left\{
    \frac{1}{2}-z + |\sin\pi z|^\alpha +|\cos k \pi z|^\alpha
    \; : \; 0\leq z \leq \frac{1}{2}
    \right\},\label{eq:3.3-u-ep-formula2}
\end{align}
and similarly,
\begin{align}
    u\left(
    \frac{1}{2},2
    \right)
    &= \min\left\{
    \frac{1}{2}-z + |\sin\pi z|^\alpha
    \; : \; 0\leq z \leq \frac{1}{2}
    \right\}.\label{eq:3.3-u-formula2}
\end{align}
By considering all the points 
where the graph of the function $z\mapsto z+$(constant) touches
the graph of $|\sin\pi z|^\alpha +|\cos k \pi z|^\alpha$
or $|\sin\pi z|^\alpha$ from below,
we obtain that the minimum of (\ref{eq:3.3-u-ep-formula2})
and (\ref{eq:3.3-u-formula2}) are
\begin{align*}
     u^{\varepsilon_k}\left(
    \frac{1}{2},2
    \right)
    &=\min\left\{
    \frac{1}{2}-z + |\sin\pi z|^\alpha
    \; : \;
    z= \frac{1}{k}\left(j+\frac{1}{2}\right),\;
    j=0,1,..., \left[\frac{k}{2}\right]
    \right\}\\
    &=\frac{1}{2}-\frac{1}{2k}
    + \left|\sin \frac{\pi}{2k}\right|^\alpha 
\end{align*}
and 
\begin{align*}
     u\left(
    \frac{1}{2},2
    \right)
    =\frac{1}{2},
\end{align*}
respectively.
Therefore, we have the convergence rate at $(x,t)=(1/2,2)$ is
\begin{align*}
    &u^{\varepsilon_k}\left(
        \frac{1}{2},2
    \right)
    -u\left(
        \frac{1}{2},2
    \right)
    =\left|
        \sin \frac{\pi}{2k}
    \right|^\alpha -\frac{1}{2k}
    =\frac{1}{k^\alpha}
    \left(
        \left|
            k \sin \frac{\pi}{2k}
        \right|^\alpha
        -\frac{1}{2k^{1-\alpha}}
    \right)\\
    &\geq\frac{1}{k^\alpha} \times \left\{
        k^\alpha
        \left(
            \frac{\pi}{2k}-\frac{1}{6}\left(
                \frac{\pi}{2k}
            \right)^3
        \right)^\alpha - \frac{1}{2k^{1-\alpha}}
    \right\}
    \geq \frac{1}{k^\alpha} \times \left\{
        \left(
            \frac{\pi}{2}-\frac{\pi^3}{48}
        \right)^\alpha 
        -\frac{1}{2}
    \right\}.
\end{align*}
This implies (\ref{eq:conv-order})
and that convergence rate of $u^\varepsilon$ is strictly
slower than $O(\varepsilon)$.
\end{example}

\begin{remark}
Here, we present a heuristic argument indicating
that $u^\varepsilon$ converges 
at the rate $O(\varepsilon^\alpha)$ in (\ref{eq:OP-ex}).
We consider (\ref{eq:OP-ex}) by separating it into two equations:
\begin{align}
    \left\{
    \begin{aligned}
        &\max \left\{u^\varepsilon_t - 1 ,\ u^\varepsilon - \psi\left(x,\frac{x}{\varepsilon}\right)\right\} =0
        && \text{in } \R \times (0,\infty),\\
        & u^\varepsilon (x,0)= u_0(x) 
        &&\text{on } \R,
    \end{aligned}
    \right.\label{eq:nucleation}
\end{align}
and
\begin{align}
    \left\{
    \begin{aligned}
        &\max \left\{u^\varepsilon_t +|u^\varepsilon_x| ,\ u^\varepsilon - \psi\left(x,\frac{x}{\varepsilon}\right)\right\} =0 
        &&\text{in } \R\times (0,\infty),\\
        &u^\varepsilon(x,0)=u_0(x)
        &&\text{on } \R.
    \end{aligned}
    \right.\label{eq:propagation}
\end{align}
Note that
if $u^\varepsilon$ satisfies (\ref{eq:propagation}),
then we have $u^\varepsilon_t \leq - |u^\varepsilon_x|\leq 0 $
and this means the graph of $u^\varepsilon$ never move to upward.
Thus, as long as $u_0\leq \psi$, we can treat (\ref{eq:propagation}) as 
\begin{align}
    \left\{
    \begin{aligned}
        &u^\varepsilon_t +|u^\varepsilon_x| =0 
        &&\text{in } \R\times (0,\infty),\\
        &u(x,0)=u_0(x)
        &&\text{on } \R.
    \end{aligned}
    \right.\label{eq:propagation-no-obs}
\end{align}
Let $S^{\varepsilon}_1(t):\BUC(\R)\to\BUC(\R)$
and $S_2^\varepsilon(t):\BUC(\R)\to\BUC(\R)$
be the solution operator of (\ref{eq:nucleation})
and (\ref{eq:propagation-no-obs}), respectively.
(i.e., for every $g\in\BUC(\R)$,
the function $(x,t)\mapsto S^\varepsilon_1(t)[g](x)$
and $(x,t)\mapsto S^\varepsilon_2(t)[g](x)$
is the unique viscosity solution to (\ref{eq:nucleation})
and (\ref{eq:propagation-no-obs})
for the initial datum $g$, respectively.)
Then, for small $h>0$,
we can interpret $u^\varepsilon$, the unique solution to (\ref{eq:OP-ex}), as
\begin{align}
    u^\varepsilon (x,t)
    \approx
    S^\varepsilon_1(t-Nh)
    \circ
    (S^\varepsilon_2(h/2)
    \circ S^\varepsilon_1(h/2))^N[u_0](x),
    \label{eq:Kato-Torotter}
\end{align}
where $N=N(t,h)\in\N$ is an integer determined by $Nh \leq t <(N+1)h$.
Moreover,
the semigroup actions $S_1^\varepsilon$ and $S_2^\varepsilon$
can be described as follows:
\begin{align}
    S_1^\varepsilon (t)[g](x)
    = \min\left\{ 
        g(x)+t,\ 
        \psi\left(
            x,\frac{x}{\varepsilon}
        \right)
    \right\},
    \label{eq:nucleation-effect}
\end{align}
and, for $S_2^\varepsilon (t)[g](x)$, 
it follows that 
\begin{align}
    \text{every level sets }
    \partial\{S^\varepsilon_2(t)[g](\cdot)>c\}
    \text{ moves by constant inward velocity } \equiv 1.
    \label{eq:propagation-effect}
\end{align}

Formal approximation (\ref{eq:Kato-Torotter}) provides an intuitive description of the evolution of the graph.
Starting from the zero initial condition, 
the solution is given by $u^{\varepsilon_k}=t$ for a short initial period.
Once the graph reaches the obstacle,
several isolated ``islands'' begin to form.
After that, every level sets of each island propagate inward with unit speed by the effect of (\ref{eq:propagation-effect}), 
while the graph is pushed upward at unit speed
by the effect of (\ref{eq:nucleation-effect}).
As time progresses, 
the effects of (\ref{eq:nucleation-effect}) and (\ref{eq:propagation-effect}) eventually balance each other,
and the graph completely stops evolving after a finite time.
(From this viewpoint, we treat (\ref{eq:OP-ex})
as a variant of birth-and-spread-type equations. 
The geometric evolution of solution graphs for such equations was investigated in \cite{MR3562367} and \cite{MR4226650}.)
\begin{figure}[h]
    \centering
\begin{tabular}{c}
\includegraphics[height=5cm]{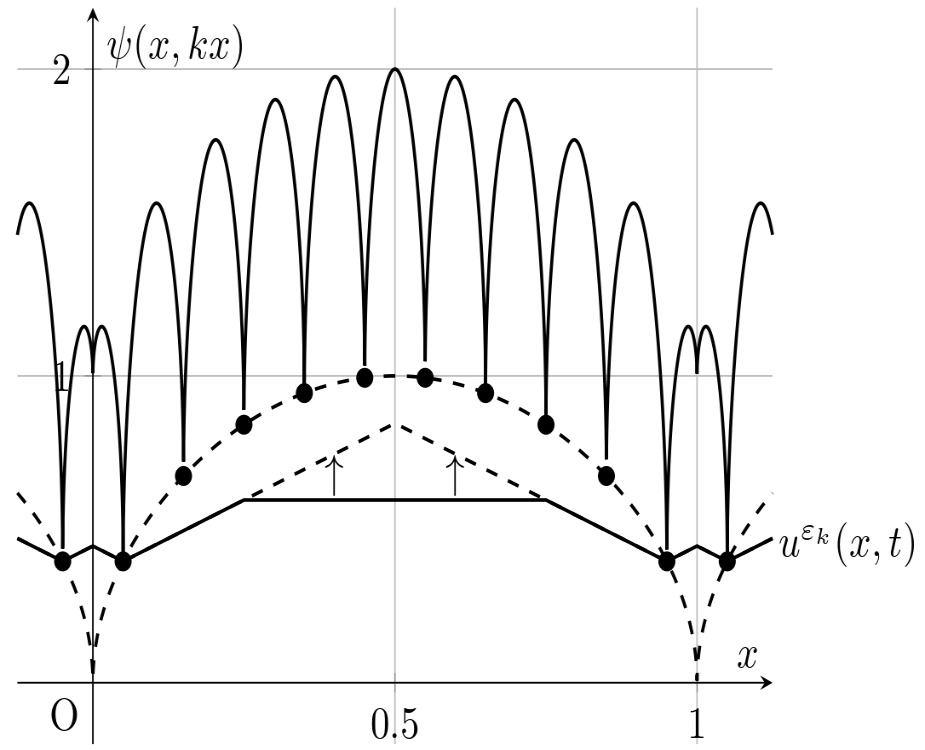} 
\end{tabular}
\caption{Evolution of $u^{\varepsilon_k}$}
\end{figure}

The height of the maximum point of the resulting stationary profile can be interpreted as the sum of two contributions:
the height at which the graph first touches the obstacle, and the height of the right isosceles triangular islands formed afterward.
These quantities are of $O(k^{-\alpha})$ and $1/2-1/k$, respectively. 
Consequently, 
the decline of the convergence rate of $u^{\varepsilon_k}-u$ is caused by the contribution of the $O(k^{-\alpha})$-scale term here.
See figure 3.

\begin{figure}[h]
    \centering
\begin{tabular}{cc}
\includegraphics[height=5cm]{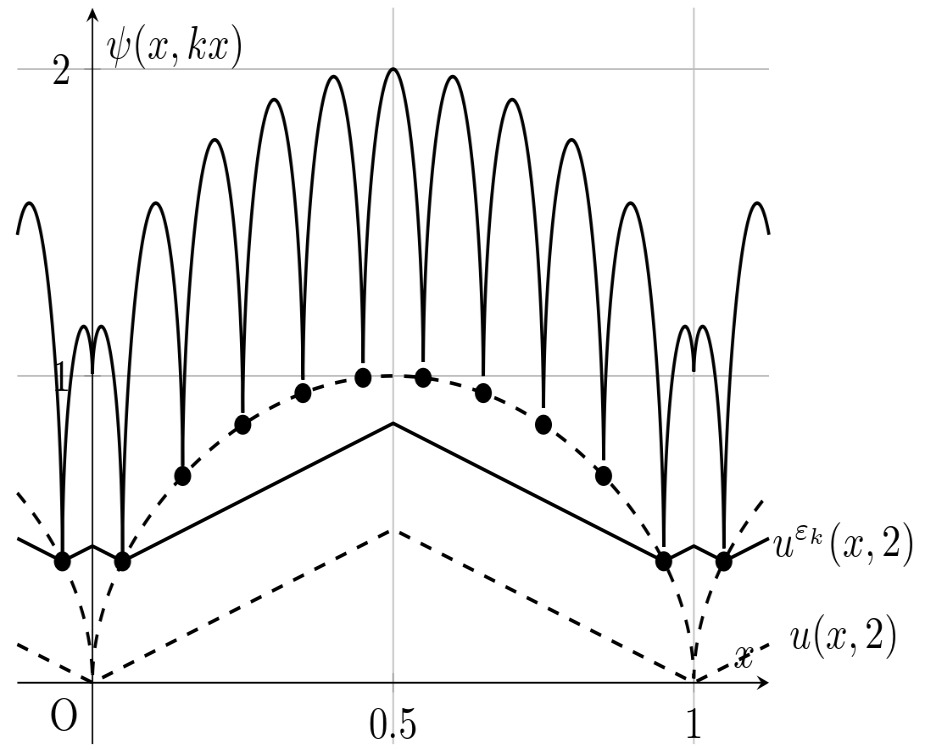} &
\includegraphics[height=5cm]{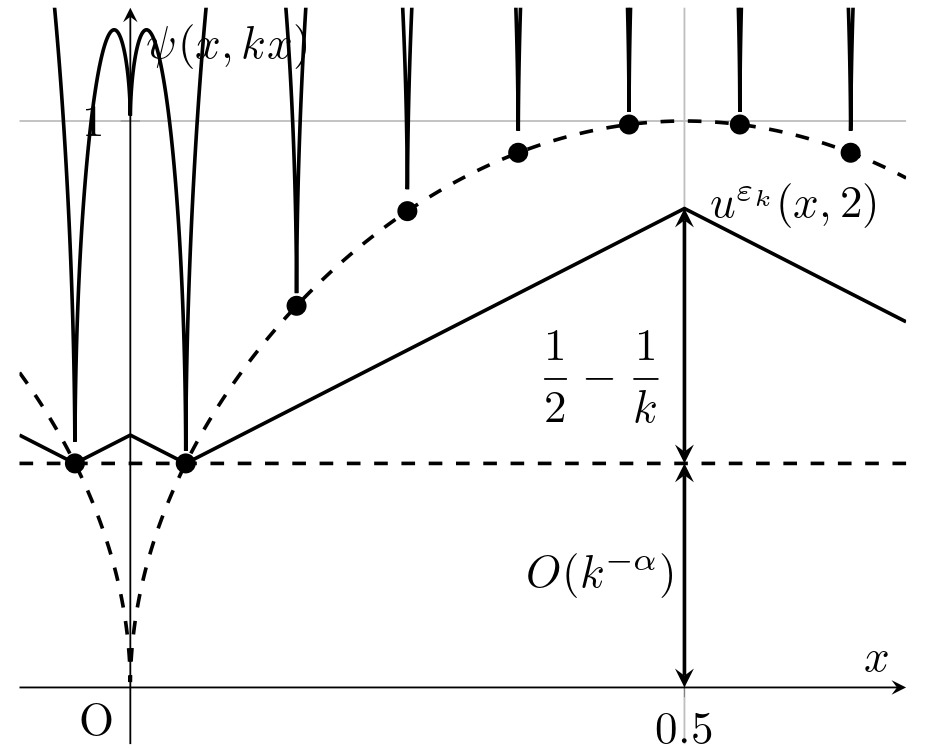}
\end{tabular}
\caption{Resulting profile of $u^{\varepsilon_k}$}
\end{figure}
\end{remark}

\section*{Acknowledgements}
The authors express their sincere gratitude
to Professor Hiroyoshi~Mitake,
Dr. Panrui~Ni,
and Professor Hung~Vinh~Tran 
for fruitful discussions and their valuable technical advice throughout this project.
They also thank
Professor Federica~Dragoni,
Professor Diogo~Gomes,
and Professor Nao~Hamamuki
for their helpful comments and encouragement.

\bibliographystyle{alpha}

\end{document}